\documentclass[10pt]{amsart}
\usepackage{color}
\usepackage{amsthm}
\usepackage{enumerate}
\usepackage{dcpic,pictex} 
\usepackage{graphicx}

\newtheorem{teo}{Theorem}[section]
\newtheorem{theorem}{Theorem}[section]
\newtheorem{lemma}[teo]{Lemma}
\newtheorem{prop}[teo]{Proposition}
\newtheorem{cor}[teo]{Corollary}

\theoremstyle{definition}
\newtheorem{defi}[teo]{Definition}
\newtheorem{definition}[teo]{Definition}
\newtheorem{example}[teo]{Example}
\newtheorem{remark}[teo]{Remark}

\newcommand{\R}{\mathbb R}

\newcommand{\C}{\mathbb C}

\newcommand{\SF}{\mathbb S}

\DeclareMathOperator{\ext}{ext}
\DeclareMathOperator{\IIm}{Im}
\DeclareMathOperator{\RRe}{Re}
\DeclareMathOperator{\ddeg}{\widetilde{deg}}

\newcommand{\N}{\mathbb N}

\newcommand{\HH}{\mathbb H}
\newcommand{\HP}{\mathbb H\mathbb{P}^1}

\theoremstyle{remark} 

\numberwithin{equation}{section}

\makeatletter 
\def\@tocline#1#2#3#4#5#6#7{\relax
  \ifnum #1>\c@tocdepth 
  \else
    \par \addpenalty\@secpenalty\addvspace{#2}%
    \begingroup \hyphenpenalty\@M
    \@ifempty{#4}{%
      \@tempdima\csname r@tocindent\number#1\endcsname\relax
    }{%
      \@tempdima#4\relax
    }%
    \parindent\z@ \leftskip#3\relax \advance\leftskip\@tempdima\relax
    \rightskip\@pnumwidth plus4em \parfillskip-\@pnumwidth
    #5\leavevmode\hskip-\@tempdima
      \ifcase #1
       \or\or \hskip 1em \or \hskip 2em \else \hskip 3em \fi%
      #6\nobreak\relax
    \dotfill\hbox to\@pnumwidth{\@tocpagenum{#7}}\par
    \nobreak
    \endgroup
  \fi}
\makeatother

\begin{document}

\title[Slice-Polynomial Functions and ruled surfaces]
{Slice-Polynomial Functions and Twistor Geometry of Ruled Surfaces in $\mathbb{CP}^3$}

\author[A. Altavilla]{A. Altavilla${}^{\dagger,\ddagger}$}\address{Altavilla Amedeo: Dipartimento Di Matematica, Universit\`a di Roma "Tor Vergata", Via Della Ricerca Scientifica 1, 00133, Roma, Italy} \email{altavilla@mat.uniroma2.it}

\author[G. Sarfatti]{G. Sarfatti ${}^{\dagger,\star}$}\address{Giulia Sarfatti: Dipartimento di Matematica e Informatica ``U. Dini'', Universit\`a di Firenze, Viale Morgagni 67/A, 50134 Firenze, Italy} \email{giulia.sarfatti@unifi.it }
\thanks{${}^{\dagger}$Partially supported by GNSAGA of INdAM, FIRB 2012 {\sl Geometria differenziale e teoria geometrica delle funzioni} and by SIR 2014 {\sl AnHyC - Analytic aspects in complex and hypercomplex geometry} n. RBSI14DYEB.   ${}^{\ddagger}$Partially supported by SIR grant {\sl NEWHOLITE - New methods in holomorphic iteration} n. RBSI14CFME. ${}^{\star}$The second author wishes to thank the Institut Montpelli\'erain Alexander Grothendieck where part of this project was carried out.}

\date{\today }


\begin{abstract} 
In the present paper we introduce the class of slice-polynomial functions: slice regular functions {defined over the quaternions, outside the real axis,} whose restriction to
any complex half-plane is a polynomial. These functions naturally emerge in the twistor interpretation of
slice regularity introduced in~\cite{gensalsto} and developed in~\cite{AAtwistor}. 
To any slice-polynomial function $P$ we associate its {\em companion} $P^\vee$ and its {\em extension} to the real axis $P_\R$, that are
quaternionic functions naturally related to $P$. Then, using the theory 
of twistor spaces, we are able to show that for any quaternion $q$ the {cardinality of simultaneous} pre-images of $q$ via $P$, $P^\vee$
and $P_\R$ is generically constant, giving a notion of degree. 
With the brand new tool of slice-polynomial functions, we {compute} the twistor discriminant locus of a cubic scroll $\mathcal{C}$ in $\mathbb{CP}^3$ 
and we conclude by giving some qualitative results on the complex structures induced by $\mathcal{C}$ via the twistor projection.

\end{abstract}
\maketitle
\tableofcontents

\section{Introduction and motivation}

%

{A fascinating aspect of the recent theory of slice regular functions over the quaternions is that it proved to have a fruitful interaction with the theory of twistor geometry. As first stated in~\cite{gensalsto} and then developed in~\cite{AAtwistor}, one can describe the theory of slice regularity using the language
	of complex geometry and thus use this description to associate with any injective slice regular function an orthogonal complex structure (OCS in the sequel) defined on its image. 
	The relevance of this construction relies on the fact that the classification of orthogonal complex structures on general open domains in $\mathbb{\R}^4$ is an open problem, \cite{salamonviac}.
	In the present work we use this twistor interpretation to analyse a special family of slice regular functions defined outside the real axis, that we named
	\textit{slice-polynomial functions}. The interest in this class of functions is motivated by the fact that they furnish a new tool to describe the twistor geometry of a wide family of ruled surfaces in $\mathbb{CP}^3$. 
	Let us begin by introducing the general contexts where we will work in.}


\subsection{Orthogonal complex structures and twistor theory of $\mathbb{S}^{4}$}
An OCS on a Riemannian manifold $(M,g)$ is an (integrable) complex structure which is Hermitian, i.e., compatible with the metric $g$.
{For our purposes we will restrict our attention to orientation preserving OCS's.}
Two global OCS's $J_{1}$, $J_{2}$ are said to be \textit{independent}
if there exists some $x\in M$ such that $J_{1}(x)\neq\pm J_{2}(x)$. Moreover they are said to be 
\textit{strongly independent} if $J_{1}(x)\neq\pm J_{2}(x)$ for any $x\in M$ (see~\cite{apostolov, pontecorvo2, pontecorvo, salamonbrno}).
It is therefore a relevant question to understand if, given a Riemannian manifold $(M,g)$, it admits any OCS and, if this is the case, to describe the
set of all OCS's defined on it.
The appropriate tool to deal with this problem is the so called \textit{twistor space} $Z$ of $M$, which is defined as the total space of the bundle containing all the OCS's defined on $M$. This space contains informations useful not only to characterise the existence of such structures (see, for instance, \cite{burstall, salamonviac}), but also to study other interesting geometric problems as the existence of minimal
surfaces (see e.g,~\cite{bryant}).

Notice that, given any OCS $J$ on a Riemannian manifold $(M,g)$, if we replace the metric $g$ by a conformally equivalent one $e^fg$, the orthogonality of $J$ is not affected. Thus, the problem of 
finding an OCS on $(M,g)$ is more related to the conformal class of $g$ than to $g$ itself. 
From this perspective we are interested in conformal transformations of $Z$ i.e. automorphisms of $Z$ induced by conformal transformations of the base space $M$.

\begin{example}
For $n>1$  the sphere $\mathbb{S}^{2n}$ has no OCS $J$ relative to its standard metric,~\cite{lebrun}. However, let $x\in \mathbb{S}^{2n}$,
then, the induced metric on $\mathbb{S}^{2n}\setminus\{x\}$ is conformally equivalent to the flat metric on $\mathbb{R}^{2n}$ and the latter 
admits infinitely many OCS's  including the constant ones parameterised by the homogeneous space $O(2n)/U(n)$. In particular, $\SF^{4}\setminus \{\infty\}$ admits only constant OCS's, parameterised by $O(4)/U(2)$.
\end{example}

For the remainder of the paper, we restrict our attention to the case of the 4-sphere $\mathbb{S}^4$ which can be identified with the quaternionic
projective line $\mathbb{HP}^1$. 
The twistor space in this case is $\mathbb{CP}^3$ and the associated bundle structure is given by $\pi:\mathbb{CP}^{3}\rightarrow \mathbb{HP}^{1}$ with fibre $\mathbb{CP}^{1}$ (see Section~\ref{lift} for further details).
It is known (see, e.g., \cite[Section 2.6]{salamonviac}), that any complex surface  in $\mathbb{CP}^{3}$
transverse to the fibres of $\pi$
{induces} an OCS on subdomains of $\mathbb{R}^{4}$ whenever such a surface is a single valued 
graph (with respect to the twistor projection). Vice versa, any OCS 
on a domain $\Omega\subset\mathbb{S}^{4}$ corresponds to a holomorphic surface in $\mathbb{CP}^{3}$.
Since it is not possible to define any
OCS on the whole $\mathbb{S}^{4}$, then no such surface in 
$\mathbb{CP}^{3}$ can intersect every fibre of the twistor fibration in exactly one point.

In this context, conformal transformations of $\mathbb{CP}^3$  are described in, e.g.,~\cite[Section 2]{armstrong}. 
In a series of papers~\cite{armstrong, APS, AS, gensalsto, salamonviac}, the authors have studied algebraic surfaces in $\mathbb{CP}^{3}$ from this point of view. 
In particular, in~\cite{salamonviac} non-singular quadric surfaces are fully classified under conformal transformations, while in~\cite{armstrong, APS, AS} qualitative and quantitative results on the
twistor geometry of non-singular cubics are given. In the paper~\cite{gensalsto} the authors study a (singular) quartic scroll by means of slice regularity. All the results mentioned in the previous papers have inferred
the existence (and in some case the explicit construction), of OCS's over explicit open dense domains of $\mathbb{S}^{4}$.

\subsection{Relation between twistor theory and slice regularity}
{The definition of slice regular functions, due to G. Gentili and D. C. Struppa, is the
most recent attempt to generalise the notion of holomorphicity to the quaternionic setting
in order to include polynomials of the form
$\sum_{n=0}^{N}q^{n}a_{n}$,
 with quaternionic coefficients, see \cite{genstru}.
The main idea is the following:
if $\mathbb{S}$ is the two dimensional sphere of quaternionic imaginary units
$\mathbb{S}:=\{q\in\mathbb{H}\,|\, q^{2}=-1\}$,
then any $q\in\mathbb{H}$ can be written as $q=\alpha+I\beta$ for $\alpha,\beta\in\mathbb{R}$ and $I\in\mathbb{S}$. 
In this way we
obtain the \textit{slice} decomposition 
$
\mathbb{H}=\bigcup_{I\in\mathbb{S}}\mathbb{C}_{I}$, where $\mathbb{C}_{I}:=Span_{\mathbb{R}}(1,I)$.
Sets of the form $\C_{I}$ are called \textit{slices} while the half-planes $\C_{I}^{+}$, identified by positive complex
imaginary part, are called \textit{semi-slices}.
For any quaternion $q=\alpha+I\beta$, its real part is the real number $\RRe(q)=\alpha$, its imaginary part is the purely imaginary number $\IIm(q)=I\beta$ and its conjugate is $\bar q =\alpha -I\beta$, so that its modulus can be computed as $|q|^2=q\bar q$.

\begin{defi}
Let $\Omega\subset\mathbb{H}$ be a domain. A differentiable function $f:\Omega\to\mathbb{H}$ is said to be 
\textit{slice regular} if for all $I\in\SF$ the restriction $f|_{\mathbb{C}_{I}}:\C_{I}\to\HH$ is a holomorphic
function with respect to the complex structure defined by quaternionic left multiplication by $I$.
\end{defi}
%
%
Several results hold for this class of functions, in analogy with the complex case: just as some examples we can mention a Cauchy integral representation formula, a Maximum and Minimum Modulus Principle, an Open Mapping Theorem. On the contrary other aspects are very different, as the nature of zeros and poles that can be either
isolated points or isolated 2-dimensional spheres. See \cite{genstostru} and references therein for an extensive account to the theory and to \cite{AAproperties, AA} for some generalizations.

Many steps forward in the theory are due to the work of R. Ghiloni and A. Perotti and their idea of using \textit{stem functions}, \cite{ghiloniperotti}. The formalism of stem functions allows, in particular, to enlarge the class of domains on which the functions can be defined (including domains that do not intersect the real axis) and to extend the theory to the more general setting of alternative algebras.

}

Let us now describe the interplay between slice regular functions and OCS's. Consider the manifold $X=\mathbb{H}\setminus\mathbb{R}$, endowed with the Euclidean metric.
Using the language of quaternions, it is possible to define a non-constant OCS on X:  
consider a point $p=\alpha+I_p\beta\in X$, with $\beta>0$, and identify $T_pX\simeq \mathbb{H}$. Then we can define the 
 OCS $\mathbb{J}$ over $X$ as
 \begin{equation*}
  \mathbb{J}_{p}:T_pX \to T_pX, \quad 
   v \mapsto \frac{\IIm(p)}{|\IIm(p)|} v = I_p v,
 \end{equation*}
where $v$ is a tangent vector to $X$ in $p$ and $I_{p}v$
denotes the quaternionic multiplication between $I_{p}$ and $v$.
It is proven in~\cite{salamonviac} that $\pm\mathbb{J}$ are the only non-constant OCS's, up to conformal
transformation of $\mathbb{S}^{4}$, that can be defined on $\mathbb{H}\setminus\mathbb{R}$.

Slice regular functions enter in the picture as follows: 
first of all $(\HH\setminus\R, \mathbb{J})$ is biholomorphic to a suitable open subset $\mathcal{Q}^+$  of the \textit{Segre Quadric} $\mathcal{Q}\subset\mathbb{CP}^{3}$.
 Let now $f:\Omega\rightarrow\mathbb{H}$  be
 a slice regular function on a circular domain (i.e. a domain which is symmetric with respect to the real axis). Then $f$ admits a {\em twistor lift}  to $\pi^{-1}(\Omega\setminus\mathbb{R})\cap \mathcal{Q}^+$,
i.e.: there exists a holomorphic
 function $\tilde{f}:\pi^{-1}(\Omega\setminus\mathbb{R})\cap \mathcal{Q}^+\rightarrow \mathbb{CP}^3$, such that $\pi\circ \tilde{f}=f\circ\pi$  (see Section~\ref{lift} for the details).  
Suppose now that the function $f$ is also injective and let $p=\alpha+I_p\beta\in\Omega \setminus \R$. Then it is possible to define an OCS $\mathbb{J}^f$ on the image of $f$ as 
\begin{equation*}
\mathbb{J}^f_{f(p)}v=\frac{\IIm(p)}{|\IIm(p)|} v = I_p v
\end{equation*}
where $\mathbb{J}^f:=(df)\mathbb{J}(df)^{-1}$ denotes the push-forward of $\mathbb{J}$ via $f$.
See~\cite{AAtwistor, gensalsto}.
In~\cite{AAtwistor} it is given a first account on the family of algebraic surfaces in $\mathbb{CP}^{3}$
that can be parameterised by the twistor lift of a slice regular function. This family is composed by
surfaces \textit{ruled by lines} called also \textit{scrolls}. In particular, this family contains
all hyperplanes and all non-singular quadrics. Moreover, up to projective transformations, for any quadric surface and cubic scroll
 $\mathcal{S}\subset\mathbb{CP}^3$ 
there exists a slice regular function $f$ such that its 
 twistor lift $\tilde{f}$ has image in $\mathcal{S}$.

An issue of this construction consists in
the impossibility, in general, of extending the twistor lift of a slice regular function to the whole quadric $\mathcal{Q}$. See for instance ~\cite[Remark 17]{AAtwistor} for an explicit simple example.
This issue is solved in the present paper in which we start the study
of what we call \textit{slice-polynomial functions}.  
\begin{definition}
	A slice regular function $P:\HH\setminus \R \to \HH$ is a slice-polynomial function if, for any $I\in\SF$, the restriction of $P$ to the semi-slice $\C_I^+$ is a polynomial. 
\end{definition}
\noindent
(See Definition~\ref{slpol} and Proposition~\ref{eachslicepol} for the details). For this class of functions we ask ourselves and answer a natural basic question: \textit{does a Fundamental Theorem of Algebra hold
for slice-polynomial functions?} The answer is positive but not trivial. In fact, we need to
consider, together with a slice-polynomial function $P$, its \textit{companion} $P^\vee$, defined as a slice-polynomial function which is a sort of dual of $P$ (see Definition~\ref{defcompanion}), and its \textit{extension} to the real axis $P_\R$ (see Subsection~\ref{extensionsection}). 
Given any quaternion $q \in \HH$, using twistor theory, we are able to study the simultaneous pre-images of $q$ via a slice-polynomial function, its companion and its extension to the reals, see Corollary \ref{mainresult}.
In particular we get the following result.  
\begin{teo}\label{mainresultintro}
	Let $P$ be a slice-polynomial function, not slice-constant and let $q$ be any quaternion. Then, generically,
	\[\# \{P^{-1}(q)\}+\# \{{P^\vee}^{-1}(q)\}+\# \{P_\R^{-1}(q)\}=d	\] 
	where $d$ is the {\em twistor degree} of $P$.
\end{teo}
\noindent
In the previous statement {\em generically} means outside a real Zariski closed subset of $\HH$.
The twistor degree of a slice-polynomial function is a suitable notion of degree in this context (see Definition~\ref{twdeg}) and furthermore it is strictly related to the degree of the algebraic surface where the twistor lift of the slice-polynomial function lies (see Corollary \ref{surf}).
As a consequence of Corollary \ref{mainresult}, we prove the following version of the Fundamental Theorem of Algebra.
	\begin{teo}\label{unionintro}
	Let $P$ be a slice-polynomial function, not slice-constant. Then
	\[P(\HH\setminus \R)\cup P^\vee(\HH\setminus \R)\cup P_\R(\R)=\HH.\]
\end{teo}

\noindent

The introduction of slice-polynomial functions provides us with a useful tool to study the twistor geometry of a certain class of surfaces in $\mathbb{CP}^3$. 
Given an algebraic surface $\mathcal{S}\subset\mathbb{CP}^{3}$ of degree $d$
its {\em twistor discriminant locus} $Disc(\mathcal{S})$ is defined as the following subset of the $4$-sphere
$$
Disc(\mathcal{S}):=\{p\in\mathbb{S}^{4}\,|\,\#(\pi^{-1}(p)\cap\mathcal{S})\neq d\}.
$$
The importance of the discriminant locus lies in the fact that its topology is a conformal invariant of the surface $\mathcal{S}$ ({see e.g.~\cite{APS,salamonviac}}). 
In particular, the number of twistor fibres contained in $\mathcal{S}$
is a conformal invariant. 
Twistor fibres can be effectively computed by means of slice regular functions:
in fact, if $f$ is a slice regular function which is constant on a $2$-sphere of the form $\alpha+\mathbb{S}\beta$, then $\pi^{-1}(f(\mathbb{S}_{\alpha+I\beta}))$ is a twistor fibre.

In this paper, as an application, we study the particular case of a cubic scroll $\mathcal{C}$ in $\mathbb{CP}^{3}$, showing some of its beautiful
topological properties. With the newly introduced theory of slice-polynomial functions we are able to compute the discriminant
locus of $\mathcal{C}$ in a couple of pages giving therefore a new method to analyse something
that is in general hard to compute (see~\cite[Section 2]{AS}).
\begin{teo}\label{discriminantintro}
	The discriminant locus of the cubic scroll 
	\[\mathcal{C}=\{[X_0,X_1,X_2,X_3] \in \mathbb{CP}^3 : X_0X_3^2+X_1^2X_2=0\}\]
	 consists in a $2$-sphere with $6$ handles, pinched at one pole and at a third root of $-1$ lying on an equator:
	\[
	Disc(\mathcal{C})=\Sigma=\hat{\C}_i\bigcup_{\substack{P=0,\infty\\ z_m^3=-1, z_m\in\C_i}}\Sigma_{P,z_m},
	\]
where $\Sigma_{P,z_m}$ denotes a handle pinched to $\hat{\C}_i$ at $P$ and $z_m$. 
\end{teo}
\noindent As we mentioned at the beginning of this introduction, there is not a complete classification of OCS's in subdomains on $\SF^4$. 
Another interesting consequence of our study, it is the fact that we are able to show the existence of $3$ different OCS's defined outside the discriminant locus of $\mathcal C$.
\begin{cor}\label{3ocsintro}
	The manifold $\SF^{4}\setminus \Sigma$ admits 3 non-constant strongly independent OCS's. 
\end{cor}

The paper is structured as follows. 
In Section 2 we introduce slice-polynomial functions with their various representations. We describe how
to extend a slice-polynomial function to the real line and we introduce the notion of {companion}. 
In Section 3 we recall the main tools in twistor geometry of the $4$-sphere and
we generalise and adapt some construction of~\cite{AAtwistor, gensalsto}. 
Both Sections 2 and 3 are quite technical but introduces several tools that are used in what follows. We point out that most of the results contained in these sections are in fact stated for a more general class of slice regular functions defined outside the real axis. 
In Section 4 we prove Theorems \ref{mainresultintro} and \ref{unionintro} using the technology of twistor spaces,
In Section 5 we use the introduced tools to determine the discriminant locus of a cubic scroll, thus proving Theorems \ref{discriminantintro} and \ref{3ocsintro}. 


\section{Slice-polynomial functions}
Let us begin by recalling some preliminary facts about the approach to slice regularity using stem functions. See \cite{ghiloniperotti} for the details. 
Let $D\subseteq \C$ be an open domain of the complex plane, such that  $D=\bar D$. In the quaternionic setting, a {\em stem function}  is a function of the form $F=F_{1}+\imath F_{2}:D\subseteq\mathbb{C}\to\mathbb{H}\otimes\mathbb{C}$,
such that, for any  $z\in D$, $F(\bar z)=\overline{F(z)}$.
Any stem function $F$ induces a {\em slice function} $f=\mathcal{I}(F):\Omega_{D}\to\mathbb{H}$, defined over the quaternionic domain $\Omega_{D}:=\{\alpha+I\beta\,|\,\alpha+i\beta\in D, I\in\mathbb{S}\}$, as $f(\alpha+I\beta)=F_{1}(\alpha+i\beta)+IF_{2}(\alpha+i\beta)$.
Sets of the form $\Omega_{D}$ are called \textit{circular domains} (in the literature they are also called \textit{slice symmetric domains} or \textit{product domains}, depending weather they intersect or not the real axis). In this paper we will always consider circular domains as domains of definition of our slice regular functions. For this reason, we will sometimes drop out the subscript $D$ to simplify the notation.
If the stem function $F$ is holomorphic, then the induced $f=\mathcal I (F)$ is a slice regular function.

A general key result in slice regularity is the \textit{Representation Formula} which allows us to 
restore a slice regular function from its values on a complex slice. 
\begin{theorem}
Let $\Omega\subseteq \HH$ be a circular domain and let $f:\Omega\to\mathbb{H}$ be a slice regular function.
	Then, for any $J\in\mathbb{S}$, $f$ is uniquely determined by its values over $\mathbb{C}_{J}$ by the following formula:
	\begin{equation*}
	f(\alpha+I\beta)=\frac{1}{2}\left[f(\alpha +J\beta )+f(\alpha -J\beta )-IJ\left(f(\alpha +J\beta )-f(\alpha -J\beta )\right)\right].
	\end{equation*}
\end{theorem}
\noindent For any $I\in\SF$, let $D\subset\C_{I}$ be a domain such that $\bar D=D$. Then any holomorphic function $f:D\to\HH$ can be uniquely extended to a slice regular function $\ext(f):\Omega_{D}\to\HH$, called
\textit{regular extension}, by means of the Representation Formula.

Given now any slice regular function on a circular domain $f:\Omega\to\HH$ it is possible to define its \textit{spherical derivative}
$\partial_{s}f:\Omega\setminus\R\to\HH$ as
$$
\partial_{s}f(q):=\frac{1}{2}\IIm(q)^{-1}(f(q)-f(\bar q)).
$$
Since $\Omega$ is circular, given any $q=\alpha+I\beta\in\Omega\setminus\R$, the 2-sphere centred at $\alpha$ with radius $\beta$ is fully contained in $\Omega$:
$$
\SF_{q}:=\{\alpha+J\beta\,|\,J\in\SF\}\subset\Omega.
$$
If $q\in\Omega\setminus\R$ is such that $\partial_{s}f(q)=0$, then $f|_{\SF_{q}}\equiv f(q)$, and
$\SF_{q}$ is said to be a \textit{degenerate sphere} for $f$. Notice that if $f=\mathcal{I}(F_1+  \imath F_2)$, then $\partial_{s}f(\alpha+I\beta)=0$ if and only if $F_2(\alpha+ i \beta)=0$. The set of degenerate spheres of 
a non-constant slice regular function $f$ is closed, with empty interior and denoted by $\mathcal{D}_{f}$.\\
As shown in~\cite{AA}, a slice regular function $f$ defined on a domain without real points
can be constant also on another type of $2$-dimensional subsets called \textit{wings}\footnote{The name ``wing'' is due to R. Ghiloni and A. Perotti.}.
A wing for a slice regular function $f$ is a surface $\Sigma\subset\Omega_{D}$ biholomorphic to $D\cap\C^{+}$ such that $f|_\Sigma$ is constant. As proved in~\cite{AA}, the set of wings of a slice regular function $f$ is also
closed and with empty interior (or the full domain if $f$ is slice-constant, see Definition~\ref{sliceconstant}).
 We will denote it by $\mathcal{W}_{f}$.
A very important topological result for slice regular function is the following.
\begin{theorem}[Open Mapping Theorem]
	Let $\Omega \subseteq \HH$ be a circular domain and let $f:\Omega\to\HH$ be any slice regular function. Then the restriction $f|_{\Omega\setminus(\mathcal{D}_{f}\cup\mathcal{W}_{f})}$ is open.
\end{theorem}
	This statement of the Open Mapping Theorem generalises the one in~\cite{genstostru}, where it is assumed
	that $\Omega\cap\R\neq\emptyset$ and the one in~\cite{AAproperties} where $\mathcal{W}_{f}$ is 
	supposed to be only composed by semi-slices of the form $\Omega\cap\C_{I}^{+}$ for suitable $I\in\SF$.
	The proof in this case can be performed suitably adapting that of~\cite[Theorem 5.1]{AAproperties}.
	Another version of the Open Mapping Theorem was recently given in~\cite{ghilperstodiv}.

Another fundamental result is the \textit{Splitting Lemma} which states that, for any $I\in\mathbb{S}$, the restriction of a slice regular function $f$ to a complex slice $\mathbb{C}_{I}$ can be written as a sum
$f|_{\mathbb{C}_{I}}=G+HJ$, where $J\in\mathbb{S}$ is orthogonal to $I$ and $G,H$ are $\mathbb{C}_{I}$-valued holomorphic functions. 

The point-wise product between two slice regular functions is not, in general, slice regular.
In the theory of slice regularity is more natural to use the $*$-product, which does preserve regularity, defined as follows. Let $f=\mathcal{I}(F)$ and 
$g=\mathcal{I}(G)$ be two slice regular functions. Then their $*$-product is defined to be the function induced by the 
point-wise product of their stem functions, i.e.:
$f*g=\mathcal{I}(FG)$.
A slice regular function $f:\Omega\to\HH$ is said to be $\C_{I}$-preserving if, for some $I\in\SF$, it holds
$f(\Omega\cap\mathbb{C}_{I})\subseteq \mathbb{C}_{I}$,
while it is called \textit{slice preserving} if the previous inclusion holds for any $I\in\SF$.
	If $f$ and $g$ are both $\C_{I}$-preserving slice regular functions then $f*g=g*f$, while if 
	$f$ is slice preserving and $g$ is any slice regular function then $f*g=g*f=fg$.

A special class of functions, introduced and studied in~\cite{AAproperties,AAtwistor,AA}, which is peculiarly quaternionic, is the class of the so called slice-constant functions.
\begin{defi}\label{sliceconstant}
A slice function $f=\mathcal{I}(F):\Omega\to\HH$ is said to be \textit{slice-constant} if its stem function $F$ is 
locally constant.
\end{defi}
If a slice-constant function is defined on a connected domain intersecting the real axis then it is constant.
As proven in~\cite[Theorem 3.4]{AAproperties}, if $f$ is a slice-constant function, then it is regular and its
{\em Cullen derivative} is identically zero (see~\cite[Definition 1.7]{genstostru}).

A fundamental example of slice-constant functions is the following.
\begin{defi}
Fix any imaginary unit $J\in\mathbb{S}$. Let $L_{+}^{J},L_{-}^{J}:\mathbb{C}\setminus \mathbb{R}\rightarrow \mathbb{H}_{\mathbb{C}}$ be the stem functions defined as
\begin{equation*}
L_{+}^{J}(z):=\begin{cases}
\frac{1- \imath J}{2},\mbox{ if }z\in\mathbb{C}^{+}\\
\frac{1+ \imath J}{2},\mbox{ if }z\in\mathbb{C}^{-}
\end{cases},\quad L_{-}^{J}(z):=\begin{cases}
\frac{1+ \imath J}{2},\mbox{ if }z\in\mathbb{C}^{+}\\
\frac{1- \imath J}{2},\mbox{ if }z\in\mathbb{C}^{-}.
\end{cases}
\end{equation*}
The previous stem functions induce the slice-constant functions,
\begin{equation*}
\ell_{+}^{J}(\alpha+I\beta)=\mathcal{I}(L_{+}^{J})(\alpha+I\beta)=\frac{1-IJ}{2},\quad \ell_{-}^{J}(\alpha+I\beta)=\mathcal{I}(L_{-}^{J})(\alpha+I\beta)=\frac{1+IJ}{2}.
\end{equation*}
\end{defi}
\noindent For $J=i$ we will use the simpler notation 
\begin{equation*}
\ell_{+}:=\ell_{+}^{i},\quad \ell_{-}:=\ell_{-}^{i}.
\end{equation*}

\begin{remark}
The family of functions just defined is of particular interest since any of its element $\ell_{+}^{J}$ is idempotent. 
More precisely, from the definition and straightforward computations, we have the following equalities:
$$
(\ell_{+}^{J})^c=\ell_{-}^{J},\qquad \ell_{+}^{J}*\ell_{+}^{J}=\ell_{+}^{J},\qquad (\ell_{+}^{J})^s=\ell_{+}^{J}*\ell_{-}^{J}\equiv 0,\qquad \ell_{+}^{J}+\ell_{-}^{J}\equiv 1.
$$
Moreover for any $J\in\SF$, the two functions $\ell_{+}^{J}$ and $\ell_{-}^{J}$ are $\C_{J}$-preserving.
\end{remark}

\begin{remark}\label{changebasis}
Given any quaternion $q\in\HH$, for any fixed $J\in\mathbb{S}$, we denote by $q^\top$ and $q^\bot$ the orthogonal projections of $q$
on $\C_J$ and $\C_J^\bot$, respectively. Therefore, since $J$ commutes with $q^\top$ and anti-commutes with $q^\bot$, for any couple $a,b\in\HH$ we have, by direct computation, that,
$$
(a^\top+a^\bot)*\ell_{+}^{J}+(b^\top+b^\bot)*\ell_{-}^{J}=\ell_{+}^{J}(a^\top+b^\bot)+\ell_{-}^{J}(a^\bot+b^\top).
$$
This family of idempotent functions generate the space of slice-constant functions in the following sense: a slice function $g$ is slice-constant if and only if for any  $J\in\mathbb{S}$, $g$ is a linear combination of the two functions $\ell_{+}^{J}$ and $\ell_{-}^{J}$, i.e. there exist $a,b\in\HH$ such that
$g=a*\ell_{+}^{J}+b*\ell_{-}^{J}=\ell_{+}^{J}(a^\top+b^\bot)+\ell_{-}^{J}(a^\bot+b^\top).
$
See~\cite[Proposition 15]{AAtwistor}.
In particular, since $\ell_+^J$ and $\ell_-^J$ are slice-constant, for any other 
$K\in\SF$, we have (in accordance with the Representation Formula)
\begin{equation}\label{changeunits}
\ell_+^J=\ell_+^K\left(\frac{1-KJ}{2}\right)+\ell_-^K\left(\frac{1+KJ}{2}\right)\quad \ell_-^J=\ell_+^K\left(\frac{1+KJ}{2}\right)+\ell_-^K\left(\frac{1-KJ}{2}\right).
\end{equation}
\end{remark}

Let $f: \Omega \to \HH$, with $\Omega \cap \R=\emptyset$, be any slice regular function. 
Thanks to the Splitting Lemma, for any $J\in\SF$,
the restriction of $f$ to $\Omega \cap \C_J$
can be written as $f_J(v)=G(v)+H(v)K$, where $K\in \SF, K \perp J$ and
$G,H: \Omega \cap \C_J \rightarrow \mathbb{C}_J$ are holomorphic functions 
of the form
\begin{equation}\label{split23}
G(v):=\begin{cases}
g(v) & v\in\Omega \cap \C_J^{+}\\
\overline{\hat{g}(\bar v)} & v\in \Omega \cap \C_J^{-}
\end{cases},\quad H(v):=\begin{cases}
h(v) & v\in \Omega \cap \C_J^{+}\\
\overline{\hat{h}(\bar v)} & v\in \Omega \cap \C_J^{-}
\end{cases},
\end{equation}
where $g,\hat{g},h,\hat{h}$ are holomorphic functions defined on $\Omega \cap \C_J^{+}$.
The expression of $G$ and $H$ in Formula~\eqref{split23} is chosen to be convenient for what follows.
Note that, since $\hat{g}$ and $\hat{h}$ are holomorphic functions, then the two functions
$v\mapsto \overline{\hat{g}(\bar v)}$, $v\mapsto\overline{\hat{h}(\bar v)}$ are holomorphic as well
and if $f_{J}$ is the restriction of a slice regular function defined over the whole algebra $\HH$
then $\overline{\hat{g}(\bar v)}=g(v)$ and $\overline{\hat{h}(\bar v)}=h( v)$.

Viceversa, given $G, H$ holomorphic functions on $\Omega\cap \C_{J}$ defined as above, we
can define a slice regular function $f:\Omega \rightarrow \mathbb{H}$ which splits over $\C_J$ as 
\begin{equation}\label{split2}
f(v)=G(v)+H(v)K=\begin{cases}
g(v)+h(v)K & v\in\Omega \cap \C_J^{+}\\
\overline{\hat{g}(\bar v)}+\overline{\hat{h}(\bar v)}K & v\in \Omega \cap \C_J^{-}
\end{cases},
\end{equation}
by means of the Representation Formula:
for any  $\alpha+I\beta\in\Omega$, with $\beta>0$,
\begin{equation}\label{splitting}
\begin{split}
f(\alpha+I\beta)  &=  \frac{1-IJ}{2}f(\alpha+J\beta)+\frac{1+IJ}{2}f(\alpha-J\beta)\\
&= \frac{1-IJ}{2}(g(\alpha+J\beta)+h(\alpha+J\beta)K)+\frac{1+IJ}{2}(\overline{\hat{g}(\alpha+J\beta)}+\overline{\hat{h}(\alpha+J\beta)}K).
\end{split}
\end{equation}
After having identified $g,h, \hat{g}$ and $\hat{h}$ with their extensions on $\C_J \setminus \R$ by Schwarz reflection (i.e. $g(\bar v)=\overline{g(v)}$), we can consider their regular extensions $\ext(g), \ext(h), \ext(\hat g), \ext(\hat h)$. Then we have that
\begin{align*}
&(\ext(g)+(\ext(\hat h))^cK)* \ell_+^J+((\ext(\hat g))^c+\ext(h)K)* \ell_-^J\\
&=\ell_+^J*(\ext(g)+\ext(h)K)+ \ell_-^J*((\ext(\hat g))^c+(\ext(\hat h))^cK).
\end{align*} 
The previous expression restricted to $\Omega \cap \C_J^+$ coincides with $g+hK$, while restricted to $\Omega \cap \C_J^-$ coincides with $\hat g+\hat hK$, therefore, thanks to Identity Principle~\cite[Theorem 3.6]{AAproperties}, we get that 
\[f=f_+*\ell^J_++f_-*\ell^J_-,\]
where 
\begin{equation}\label{effe+}
f_+=\ext(g)+(\ext(\hat h))^cK \quad \text{and} \quad f_-=(\ext(\hat g))^c+\ext(h)K.
\end{equation}
Summarising we have proved the following result.
\begin{teo}
	Let $f: \Omega \to \HH$ be a slice regular function on a circular domain $\Omega$ such that $\Omega \cap \R=\emptyset$. Then, for any $J\in \SF$, there exist and are unique $f_+,f_-:\Omega \to \HH$ slice regular functions, such that $f=f_+*\ell_+^J+f_-*\ell_-^J$.
	
\end{teo}

\begin{remark}
The previous theorem is the restatement, in the context of slice regularity, of the well-known \textit{Peirce decomposition} (see~\cite[Chapter 7, \S 21]{lam}). 
\end{remark}

\begin{remark}\label{obvdec}
If $\Omega\cap\R\neq\emptyset$, then, given any slice regular function $f:\Omega\to\R$, its restriction
$f|_{\Omega\setminus\R}$ can be written as
$$
f|_{\Omega\setminus\R}=f*\ell_{+}^{J}+f*\ell_{-}^{J}=f*(\ell_{+}^{J}+\ell_{-}^{J}),
$$
for any $J\in\SF$.
\end{remark}

The class of functions that, as we will see, furnishes a new tool in twistor geometry is the following. 

\begin{defi}[slice-polynomial functions]\label{slpol}
A slice regular function $P:\HH\setminus\R\to\HH$ is said to be a \textit{slice-polynomial function} if there exist $J\in\SF$ and two 
quaternionic polynomials $P_+,P_-:\HH\to\HH$ such that
$$
P=P_+*\ell_+^J+P_-*\ell_-^J.
$$
\end{defi}

\noindent The next proposition shows that Definition~\ref{slpol} is well posed.

\begin{prop}\label{slicepolprop}
Let $P:\HH\setminus\R\to\HH$ be a slice-polynomial function, then for any $K\in\SF$, there exist two 
quaternionic polynomials $Q_+,Q_-:\HH\to\HH$ such that
$$
P=Q_+*\ell_+^K+Q_-*\ell_-^K.
$$
\end{prop}

\begin{proof}
Thanks to Remark~\ref{changebasis}, for any $K\in\SF$, 
there exist $a,b,c,d\in\HH$ such that
$$
\ell_+^J=a*\ell_{+}^{K}+b*\ell_{-}^{K},\quad\ell_-^J=c*\ell_{+}^{K}+d*\ell_{-}^{K}.
$$
Thanks to Equation~\eqref{changeunits}, by standard computations, $a=d$ and $b=c$, therefore 
\begin{align*}
P&=P_+*\ell_+^J+P_-*\ell_-^J\\
&= P_+*(a*\ell_{+}^{K}+b*\ell_{-}^{K})+P_-*(b*\ell_{+}^{K}+a*\ell_{-}^{K})\\
&=(P_+a+P_-b)*\ell_+^K+
(P_+b+P_-a)*\ell_-^K.
\end{align*}

\end{proof}

\noindent As a consequence, we get that the restriction of a slice-polynomial function to any semi-slice is a polynomial and moreover we have the following characterization.

\begin{prop}\label{eachslicepol}
Let $P:\HH\setminus\R\to\HH$ be a slice regular function. $P$ is a slice
polynomial if and only if for any $K\in\SF$ the restriction 
$P|_{\C_K^+}$ is a polynomial in the variable $\alpha+K\beta\in\C_K^+$.
\end{prop}

\begin{proof}

First of all notice that, if $q=\alpha+I\beta\in\HH$ and $v=\alpha+i\beta\in\C_{i}$, then, by direct computation
$$
q*\ell_{+}=\ell_{+}v,\qquad q*\ell_{-}=\ell_{-}\bar v.
$$
Therefore, for any $k,h\in\N$ and any $a,b\in\HH$ we have that 
\begin{eqnarray*}
(q^{k}a)*\ell_{+}+(q^{h}b)*\ell_{-} & = & (q^{k}a^{\top}+q^{k}a^{\bot})*\ell_{+}+(q^{h}b^{\top}+q^{h}b^{\bot})*\ell_{-}\\
&=& \ell_{+}(v^{k}a^{\top}+v^{h}b^{\bot})+\ell_{-}(\bar v^{k}a^{\bot}+\bar v^{h}b^{\top}).
\end{eqnarray*}
Using the last computation, if $P$ is a slice-polynomial function, we obtain the following equality
\begin{eqnarray*}
P(q) & = &\left(\sum_{k=0}^{N}q^{k}a_{k}\right)*\ell_{+}+\left(\sum_{h=0}^{M}q^{h}b_{h}\right)*\ell_{-}\\
& = & \left(\sum_{k=0}^{N}q^{k}a_{k}^{\top}+q^{k}a_{k}^{\bot}\right)*\ell_{+}+\left(\sum_{h=0}^{M}q^{h}b_{h}^{\top}+q^{h}b_{h}^{\bot}\right)*\ell_{-}\\
& = & \ell_{+}\left(\sum_{k=0}^{N}v^{k}a_{k}^{\top}+\sum_{h=0}^{M}v^{h}b_{h}^{\bot}\right)+\ell_{-}\left(\sum_{k=0}^{N}\bar v^{k}a_{k}^{\bot}+\sum_{h=0}^{M}\bar v^{h}b_{h}^{\top}\right).
\end{eqnarray*}

\noindent Therefore, thanks to Proposition~\ref{slicepolprop}   and the fact that, for any $K\in\SF$,
$\ell_+^K\equiv 1$ and $\ell_-^K\equiv 0$ on $\C_K^+$, 
we immediately have the thesis. The opposite implication follows by reading the previous chain of equalities in the opposite direction.
\end{proof}

Thanks to Remark~\ref{obvdec}, any quaternionic polynomial $P:\mathbb{H}\to\HH$ is a slice-polynomial function if restricted to $\HH\setminus\R$. In fact, it holds
$$
P|_{\HH\setminus \R}=P*\ell_{+}+P*\ell_{-}.
$$

\begin{example}\label{firstexample}
It is easy to generate examples of slice-polynomial functions. However we recall here some 
slice-polynomial functions naively introduced in past researches by the first author~\cite{AAproperties, AA, AAtwistor}.
In~\cite{AAproperties} the function $P(q)=q*\ell_+$ was introduced and analysed for its topological properties,
showing in particular that it admits the whole semi-slice $\C_i^-$ as zero set. Then in~\cite{AA}
slice regular functions defined over circular domains non intersecting the real axis were studied from a differential point of view
and it was shown that the function $Q(q)=(q+j)*\ell_{+}$ admits two non-compact surfaces, both
biholomorphic to the half-plane $\C^+$, where it is constant. Finally in~\cite{AAtwistor}
the twistor geometry induced by $P(q)$ was studied and the function $R(q)=-q^{2}*\ell_{+}+q*\ell_{-}$ 
was proposed as an example of slice regular function whose twistor lift parameterises a cubic scroll.
All the functions listed in this example are slice-polynomial functions.
\end{example}

Thanks to Proposition~\ref{slicepolprop}, from now on, if not differently specified, we will always write slice-polynomial functions with respect to the imaginary unit $i$:
\begin{equation}\label{slicepol}
P=P_+*\ell_++P_-*\ell_-.
\end{equation}
We want now to express slice-polynomial functions in terms of the Splitting Lemma.

If $P$ is a slice-polynomial function which splits on $\C_i$ as 
\begin{equation}\label{standardsplitting}
P(v):=\begin{cases}
g(v)+h(v)j & v\in\C_{i}^{+}\\
\overline{\hat{g}(\bar v)}+\overline{\hat{h}(\bar v)}j & v\in \C_{i}^{-},
\end{cases}
\end{equation}
we have that $g,\hat{g},h,\hat{h}$ are polynomials with coefficients in $\C_i$.
\noindent Moreover, from the computations in the proof of Proposition~\ref{eachslicepol} and Equation~\eqref{splitting} we have that 
\begin{equation*}
\begin{split}
P(\alpha+I\beta)&=\frac{1-Ii}{2}(g(\alpha+i\beta)+h(\alpha+i\beta)j)+\frac{1+Ii}{2}(\overline{\hat{g}((\alpha+i\beta)}+\overline{\hat{h}((\alpha+i\beta)}j)\\
&= ((g+\hat{h}^cj)*\ell_++(\hat{g}^c+hj)*\ell_-)(\alpha+I\beta)
\end{split}
\end{equation*}
where in the last equality, with a slight abuse of notation, we are identifying $g,\hat{g},h,\hat{h}$ with their regular extensions.
Comparing the last representation with Formula~\eqref{slicepol}, we obtain that $P_+=g+\hat{h}^cj$
and $P_-=\hat{g}^c+hj$.

\begin{example}
Let us consider the polynomial $P(q)=q^{2}+qi$, then $h=\hat h\equiv 0$ and $g(v)=\overline{\hat{g}(\bar v)}=v^{2}+vi$. Therefore $\overline{\hat{g}(v)}=\bar v^{2}+\bar v i$
$$
P(q)=q^{2}+qi=\ell_{+}(v^{2}+vi)+\ell_{-}(\bar v^{2}+\bar v i).
$$
The three slice-polynomial functions introduced in Example~\ref{firstexample} can be written as follows
$$
P(q)=\ell_+ v,\qquad Q(q)=\ell_+ v+\ell_- j,\qquad R(q)=\ell_+(-v^2)+\ell_-\bar v.
$$
\end{example}

\subsection{Extension to the real line}\label{extensionsection}

Given a slice-polynomial function $P=\mathcal{I}(F_{1}+\imath F_{2})$ we want to analyse its behaviour approaching the real line. Of course, since for any $I\in\mathbb{S}$ the function $P|_{\mathbb{C}_{I}^{+}}$ is a polynomial, then it can be extended to $\mathbb{R}$ but the values that it reaches strongly depend on $I$.
The right set to consider is therefore $\mathbb{S}\times \mathbb{R}$ and not merely $\mathbb{R}$.
An element $(I,\alpha)\in\mathbb{S}\times\mathbb{R}$ will also be denoted by $\alpha_{I}$.
We define then the following function
\begin{equation*}
P_{\mathbb{R}}:\mathbb{S}\times\mathbb{R}\to\mathbb{H}
\end{equation*}
\begin{equation}\label{pr}
 P_{\mathbb{R}}(\alpha_{I}) =P_{\R}(I,\alpha):=\lim_{\beta\to 0^+}P(\alpha+I\beta) =\lim_{\beta\to 0^+}F_{1}(\alpha+i\beta)+IF_{2}(\alpha+i\beta)=F_{1}(\alpha)+IF_{2}(\alpha)
\end{equation}
%
%
where the limit always exists, since $P|_{\mathbb{C}_{I}^{+}}$ is a polynomial on $\mathbb{C}_{I}^{+}$.
We define, then, the following sets
\begin{equation*}
P_{\mathbb{R}}(\alpha):=\bigcup_{I\in\mathbb{S}}P_{\mathbb{R}}(\alpha_{I}),\quad
P_{\mathbb{R}}(\mathbb{R}):=P_{\mathbb{R}}(\mathbb{S}\times\mathbb{R})
\end{equation*}

Notice that, if $P=\mathcal{I}(F_{1}+\imath F_{2})$ is a slice-polynomial function and there exists $\alpha\in\R$ such that $P(\alpha)=\{q\}$ for some $q\in\HH$, then, for any $I,J\in\SF$,
\[F_1(\alpha)+IF_2(\alpha)=q=F_1(\alpha)+JF_2(\alpha),\]
which entails that $F_2(\alpha)=0$.
For this reason the slice-polynomial function $P$ can be extended to $\alpha\in\R$ as $P(\alpha)=q$.

If $P$ is a quaternionic polynomial, then, thanks to the odd character of $F_{2}$ with respect to $\beta$,
for any $\alpha\in\mathbb{R}$, $F_{2}(\alpha)=0$
and so $P_{\mathbb{R}}(\alpha)=\{P(\alpha)\}$ and $P_{\mathbb{R}}(\mathbb{R})=P(\mathbb{R})$.
If $P$ is not a quaternionic polynomial, then, $P_{\R}(\alpha)$
can be of dimension $2$. 

\begin{example}
Consider the slice-polynomial function $P(q)=q*\ell_+$, or, explicitly, $P(\alpha+I\beta)=(\alpha+I\beta)*\ell_+=(\alpha+I\beta)\frac{1-Ii}{2}$. 
Therefore $P(\alpha+I\beta)=\frac{1}{2}(\alpha+I\beta-\alpha Ii+\beta i)$ and for any $I\in\SF$ and any $\alpha\in\R$
$$
P_{\R}(\alpha_{I})=\lim_{\beta\to0^+}P(\alpha+I\beta)=\alpha\frac{(1-Ii)}{2}.
$$
If $I=Ai+Bj+Ck$, then  $P|_{\R}(\alpha_{I})$ belongs to the real hyperplane $\{q_0+q_1i+q_2j+q_3k\in\HH\,|\,q_1=0\}$.
Hence, for any $I\in\SF\setminus\{-i\}$, $P|_{\mathbb{C}_{I}^{+}}(\mathbb{R})$ is a real line and for any $\alpha\neq 0$,
$P_\R(\alpha)$ is a 2-sphere. It is not difficult to see that $P_\R(\R)$ covers the whole hyperplane $\{q_0+q_1i+q_2j+q_3k\in\HH\,|\,q_1=0\}$ 
(see~\cite[Theorem 34]{AAtwistor}).

With similar computations, if $Q(q)=(q+j)*\ell_{+}$, $R(q)=-q^{2}*\ell_{+}+q*\ell_{-}$ and $q=\alpha+I\beta$, then 
$$
Q|_{\R}(\alpha_{I})=\lim_{\beta\to0^+}Q(\alpha+I\beta)=\alpha\frac{(1-Ii)}{2}+\frac{(1+Ii)}{2}j.
$$
and
$$
R|_{\R}(\alpha_{I})=\lim_{\beta\to0^+}R(\alpha+I\beta)=-\alpha^2\frac{(1-Ii)}{2}+\alpha\frac{(1+Ii)}{2}.
$$
For any fixed $I\in\SF\setminus\{-i\}$ the set $Q_\R(\alpha_I)$ parameterises an affine line in $\HH$ passing through the point $\frac{(1+Ii)}{2}j$,
therefore the resulting set $Q_\R(\R)$ is a ruled $3$-dimensional manifold. Finally, for any $I\in\SF$, $R_\R(\alpha_I)$ parameterises a conic
in $\HH$. In particular $R_{\R}(\alpha_{i})$ is a parabola while 
$R_{\R}(\alpha_{-i})$ is a straight line, both  passing through the origin.
\end{example}

\begin{remark}\label{parterealedegenere}
Given a slice-polynomial function $P=\mathcal{I}(F_{1}+\imath F_{2})$,  the sphere $\SF_{v}$ is degenerate for $P$ if and only if $\partial_{s}P(v)=0$. Writing the quantity $P(v)-P(\bar v)$ in terms of the splitting~\eqref{standardsplitting} we get, 
$$
\overline{\hat{g}(v)}+\overline{\hat{h}(v)}j-g(v)-h(v)j=0,
$$
that is equivalent to the system
\begin{equation}\label{systemf2}
\begin{cases}
\overline{\hat{g}(v)}-g(v)=0\\
\overline{\hat{h}(v)}-h(v)=0.
\end{cases}
\end{equation}
Since $g,h,\hat g$ and $\hat h$ are polynomials, the last notion can be extended to $\mathbb{R}$
(and so to $P_{\mathbb{R}}$), saying that a real point $\alpha$ is degenerate for $P$ if and only if
$P_{\mathbb{R}}(\alpha)=\{q\}$ i.e. $P$ can be defined at $\alpha$ as $P(\alpha)=q$.
In particular, with this notion, any quaternionic polynomial $P$ defined on $\mathbb{H}$ is such that
the real line is degenerate.
\end{remark}
\begin{remark}
We point out that all the results proved in this subsection can be restored for any 
slice regular function suitable defined, i.e. for any slice regular function  $f=\mathcal{I}(F_{1}+\imath F_{2})$ defined on a circular domain $\Omega \subseteq \HH\setminus\R$ not intersecting the real axis, such that $F_{1}$ and $F_{2}$ extend continuously to $\R$.
\end{remark}

\subsection{Companion of a slice regular function} 
Let us introduce the notion of \textit{companion} of a slice regular function. Given a slice regular function $f$, its companion is in some sense dual to $f$ and, besides its algebraic definition, we will see that it naturally arises in the geometric construction described in Section \ref{lift}. 

\begin{definition}\label{defcompanion}
Let $f:\Omega\to\HH$ be a slice regular function and suppose that $f|_{\Omega \setminus \R}=f_+*\ell_++f_-*\ell_-$. We define the \textit{companion} of $f$ as the slice regular function $f^\vee:\Omega \setminus \R\to\HH$, defined by
$f^\vee=f_-*\ell_++f_+*\ell_-$.
\end{definition}
\begin{remark}
The previous definition is well posed. Suppose in fact that 
$f=f_+*\ell_++f_-*\ell_-=g_{+}*\ell_{+}^{J}+g_{-}*\ell_{-}^{J}$, for some $J\in\SF$, $J \neq i$. Then,
if $\ell_{+}=a*\ell_{+}^{J}+b*\ell_{-}^{J}$ (and so $\ell_{-}=b*\ell_{+}^{J}+a*\ell_{-}^{J}$), for suitable
$a,b\in\HH$, we have that
\begin{align*}
f&=f_+*\ell_++f_-*\ell_-=f_+*(a*\ell_{+}^{J}+b*\ell_{-}^{J})+f_-*(b*\ell_{+}^{J}+a*\ell_{-}^{J})\\ &=[f_{-}b+f_{+}a]*\ell_{+}^{J}+[f_{-}a+f_{+}b]*\ell_{-}^{J}=g_{+}*\ell_{+}^{J}+g_{-}*\ell_{-}^{J}
\end{align*}
and hence that 
$$
f^{\vee}=[f_{-}a+f_{+}b] *\ell_{+}^{J}+[f_{-}b+f_{+}a]*\ell_{-}^{J}=g_{-}*\ell_{+}^{J}+g_{+}*\ell_{-}^{J}.
$$
\end{remark}
Thanks to Remark~\ref{obvdec}, if $\Omega\cap\R\neq \emptyset$, then the companion $f^{\vee}$ of $f:\Omega\to\HH$ is
$f$ itself, i.e.: $f^{\vee}=f$.

%

\begin{remark}\label{checkJ}
Let $f:\Omega\to\HH$ be any slice regular function that splits on $\C_{J}$ as
$f|_{\C_{J}}=G(v)+H(v)K$, as in Formula~\eqref{split2}. Then, thanks to Equations \eqref{splitting} and \eqref{effe+}, we get that
its companion $f^\vee:\Omega\to\HH$ splits on $\mathbb{C}_J$  as
\begin{equation*}
f^\vee(v):=\begin{cases}
\overline{\hat{g}(\bar v)}+\overline{\hat{h}(\bar v)}K & v\in\C_{J}^{+}\\
g(v)+h(v)K  & v\in \C_{J}^{-}.
\end{cases}
\end{equation*}
In particular  if $P$ is a slice-polynomial function, then its companion $P^{\vee}$ is a slice-polynomial function, and 
if $P$ can be written at $\alpha+I\beta$ as $P(\alpha+I\beta)=\ell_+^Jp_+(\alpha+J\beta)+\ell_-^Jp_-(\alpha-J\beta)$, then
$P^\vee=\ell_+^Jp_-(\alpha+J\beta)+\ell_-^Jp_+(\alpha-J\beta)$.
\end{remark}

\begin{prop}\label{degeneratecoincide}
Let $f$ be any slice regular function. Then the degenerate sets of $f$ and $f^{\vee}$ coincide.
\end{prop}

\begin{proof}
Consider the splitting of $f$ defined as in Equation~\eqref{split2}. Let $\mathbb{S}_{q}$ be a sphere in the degenerate set of $f$ and let $\{v,\bar v\}=\SF_{q}\cap \C_{J}$. Then System~\eqref{systemf2} holds for $v$ and $\bar v$. Since 
\begin{equation*}
f^\vee(v):=\begin{cases}
\overline{\hat{g}(\bar v)}+\overline{\hat{h}(\bar v)}K & v\in\C_{J}^{+}\\
g(v)+h(v)K  & v\in \C_{J}^{-},
\end{cases}
\end{equation*}
it is immediate to check that $f^{\vee}(v)=f^{\vee}(\bar v)$ and hence that $\mathbb{S}_{q}$ is a sphere in the degenerate set of $f^{\vee}$.
\end{proof}

\begin{prop}
For any slice-polynomial function $P$ we have that $P_{\R}(\R)=P_{\R}^\vee(\R)$.
\end{prop}
\begin{proof}
Fix any $J\in\SF$, then, thanks to Remark~\ref{checkJ}, we have that 
$$
P_{\R}(\alpha_{J})=\lim_{\beta\to 0^+}P_+(\alpha+J\beta)=P_{\R}^\vee(\alpha_{-J}).
$$
\end{proof}

\begin{remark}
As before, the last proposition can be restored for any 
slice regular function suitable defined, i.e. for any  $f=\mathcal{I}(F_{1}+\imath F_{2}):\Omega \subseteq \HH\setminus\R\to\HH$
such that $F_{1}$ and $F_{2}$ extend continuously to $\R$.
\end{remark}

\section{Twistor fibration and blow up of the real line}\label{lift}

In this section we recall some fact concerning the twistor interpretation of slice regular functions and we extend some known results contained in ~\cite{AAtwistor,gensalsto}.

Consider the \textit{left} quaternionic projective space $\HP$ as the set of equivalence classes 
$[q_1,q_2]\sim[pq_1,pq_2]$, for any $p\in\mathbb{H}\setminus\{0\}$. The twistor fibration: $\pi:\mathbb{CP}^3\rightarrow\mathbb{HP}^1$ is defined as $\pi[X_0,X_1,X_2,X_3]=[X_0+X_1j,X_2+X_3j]$. We embed $\mathbb{H}$ in $\mathbb{HP}^1$ as $q\mapsto [1,q]$. So that if $q=q_1+q_2j\in\mathbb{H}$ and $q_1,q_2\in\mathbb{C}_i$, then
\begin{equation*}
 \pi[X_0,X_1,X_2,X_3]=[1,q]\,\Leftrightarrow\, [1,(X_0+X_1j)^{-1}(X_2+X_3j)]=[1,q_1+q_2j],
\end{equation*}
and the last equality is equivalent to the following system which, therefore, describes the fibers of $\pi$
\begin{equation}\label{eqfiber}
 \begin{cases}
  X_2=X_0q_1-X_1\bar q_2\\
  X_3=X_0q_2+X_1\bar q_1.
 \end{cases}
\end{equation}

Let $\mathcal{Q}\simeq(\mathbb{CP}^1\times\mathbb{CP}^1)\subset\mathbb{CP}^3$ be the non-singular {\em Segre quadric} defined by the equation $X_0X_3=X_1X_2$, and
define the {\em slice complex structure} over $\mathbb{H}\setminus \mathbb{R}$ as $\mathbb{J}_{\alpha+I\beta}(v)=Iv$, for $\beta>0$ and for any $v\in T_{\alpha+I\beta}(\HH\setminus\R)\simeq \HH$.
Then, as proven in \cite{gensalsto}, the quadric $\mathcal{Q}$ contains two open subset  $\mathcal{Q}^+$ and  $\mathcal{Q}^-$ both
 biholomorphic, via the map $\pi$, to $(\mathbb{H}\setminus\mathbb{R},\mathbb{J})$.
The set $\mathcal{Q}^+\simeq (\mathbb{CP}^1\times\C^+)$ consists of points for which at least one of the following holds
 \begin{itemize}
 \item $X_0\neq 0$ and $X_2/X_0 \in\mathbb{C}^+$;
 \item $X_1\neq 0$ and $X_3/X_1 \in\mathbb{C}^+$. 
 \end{itemize}
 Analogously $\mathcal{Q}^-\simeq (\mathbb{CP}^1\times\C^-)\subset \mathcal{Q}$ is defined as the set of points  for which at least one of the following holds
 \begin{itemize}
 \item $X_0\neq 0$ and $X_2/X_0 \in\mathbb{C}^-$;
 \item $X_1\neq 0$ and $X_3/X_1 \in\mathbb{C}^-$. 
 \end{itemize}
%
%
To find coordinates for the fiber $\pi^{-1}(\alpha+I\beta)$ intersected with $\mathcal{Q}$ we will suppose $X_0\neq 0$. Denote then $I=Ai+Bj+Ck$ and $v=\alpha+i\beta$, with $\beta>0$.
Let $q=\alpha+I\beta=q_1+q_2j$, with $q_1=\alpha+\beta Ai$ and $q_2=\beta B+\beta Ci$.
Then the system defining the intersection,
\begin{equation*}
 \begin{cases}
  X_0X_3=X_1X_2\\
  X_2=X_0q_1-X_1\bar q_2\\
  X_3=X_0q_2+X_1\bar q_1
 \end{cases}
\end{equation*}
becomes 
\begin{equation*}
 X_0^2q_2+X_0X_1[\bar q_1 -q_1]+X_1^2\bar q_2=0
\end{equation*}
and imposing $X_0=1$, we get
\begin{equation*}
 q_2+X_1[\bar q_1 -q_1]+X_1^2\bar q_2=0.
\end{equation*}
Solving the last equation we find $X_1\in\left\{u=-i\frac{B+iC}{1+A}, w=i\frac{B+iC}{1-A}\right\}$, and hence
$$
\pi^{-1}(\alpha+I\beta)\cap \mathcal{Q}^+=[1,u,v,uv],\qquad \pi^{-1}(\alpha+I\beta)\cap \mathcal{Q}^-=[1,w,\bar v,w\bar v].
$$
%
%
%
Analogously we get
$$
\pi^{-1}(\alpha-I\beta)\cap \mathcal{Q}^+=[1,w,v,wv],\qquad\pi^{-1}(\alpha-I\beta)\cap \mathcal{Q}^-=[1,u,\bar v,u\bar v].
$$
Notice that $w=-\bar u^{-1}$.

%
%

Consider now a slice regular function $f:\HH\setminus\R\to\HH$ which splits on $\C_i$ as  in Formula \eqref{split2}.
As proven in \cite{AAtwistor,gensalsto}, $f$ can be lifted on $\mathcal{Q}^+$ with the following non-homogeneous parameterisation:
\begin{equation}\label{sollevamento+}
[1,f(q)] = \pi[\tilde{f}[\pi^{-1}(q)]]=\pi[\tilde{f}[1,u,v,uv]]=\pi[1,u,g(v)-u\hat{h}(v),h(v)+u\hat{g}(v)].
\end{equation}

With similar computations as in~\cite{AAtwistor,gensalsto} it is possible to prove that any slice regular function $f:\Omega\to\HH$ defined
on a circular domain $\Omega\subseteq \HH \setminus\R$, can be lifted to a holomorphic map $\tilde{f}_-:\mathcal{Q}^-\cap\pi^{-1}(\Omega)\to\mathbb{CP}^3$. The aim of the 
following lemma is to exhibit an explicit parameterisation for $\tilde{f}_-$.
\begin{lemma}\label{lemmaqminus}
Let $\Omega \subseteq \HH\setminus \R$ and let $f:\Omega\to\HH$ be a slice regular function that splits on $\C_i$ as in Formula~\ref{split2}. Then, the parameterisation of the twistor lift
$\tilde{f}_-:(\mathcal{Q}^-\cap\pi^{-1}(\Omega))\to\mathbb{CP}^3$ has coordinates
$$
\tilde{f}_-[1,w,v,wv]=[1,w,\overline{\hat{g}(v)}-w\overline{h(v)},\overline{\hat{h}(v)}+w\overline{g(v)}].
$$
\end{lemma}
\begin{proof}
Let $F_1+\imath F_2$ be the stem function inducing $f$, and let $q=\alpha+I\beta\in\Omega$, $v=\alpha+i\beta\in\Omega\cap\C_i$ and $Q_w=1+wj$
be so that $\alpha+I\beta=Q_w^{-1}\bar v Q_w$.
In the same spirit of~\cite[Theorem 5.3]{gensalsto} and of~\cite[Theorem 24]{AAtwistor}, the thesis is a consequence of the following
sequence of equalities.
\begin{equation*}
  \begin{array}{rcl}
   [1,f(q)] & = & [1,f(Q_w^{-1}(\alpha-i\beta)Q_w)]\\
   & = & [1,f(\alpha-Q_w^{-1}(i)Q_w\beta)]\\
   & = & [1,F_1(\alpha-i\beta)+Q_w^{-1}(i)Q_wF_2(\alpha-i\beta)]\\
   & = & [Q_w,Q_wF_1(\alpha-i\beta)+iQ_wF_2(\alpha-i\beta)]\\
   & = & [1+wj,(1+wj)F_1(\alpha-i\beta)+i(1+wj)F_2(\alpha-i\beta)]\\
   & = & [1+wj,f(\alpha-i\beta)+wjf(\alpha+i\beta)]\\
   & = & [1+wj,f(\bar v)+wjf(v)]\\
   & = & [1+wj,\overline{\hat{g}(v)}+\overline{\hat{h}(v)}j+wj(g(v)+h(v)j)]\\
   & = & \pi[1,w,\overline{\hat{g}(v)}-w\overline{h(v)},\overline{\hat{h}(v)}+w\overline{g(v)}].
  \end{array}
 \end{equation*}
\end{proof}

Both Theorems 5.3 in~\cite{gensalsto} (or its extension~\cite[Theorem 24]{AAtwistor}), and Lemma~\ref{lemmaqminus} give a result that in the case of functions defined outside the real axis
is only {half-satisfactory}, in the following sense. Consider a slice regular function $f:\HH\setminus\R\to\HH$ that splits as in Formula~\eqref{split2}
such that $\overline{\hat{g}(\bar v)}\neq g(v)$ or $\overline{\hat{h}(\bar v)}\neq h( v)$. Then the commutation of the following diagrams given in~\cite{AAtwistor, gensalsto}
does not extends to $\mathcal{Q}$ (see~\cite[Section 6]{AAtwistor}).

$$
\begindc{\commdiag}[50]
\obj(-10,10)[aa]{$\mathcal{Q}^+$}
\obj(10,10)[bb]{$\tilde{f_+}(\mathcal{Q}^+)$}
\obj(-10,0)[cc]{$\HH\setminus\R$}
\obj(10,0)[dd]{$f(\HH\setminus\R)$}
\mor{aa}{bb}{$\tilde{f_+}$}
\mor{aa}{cc}{$\pi$}[-1,0]
\mor{bb}{dd}{$\pi$}
\mor{cc}{dd}{$f$}
\obj(30,10)[aaa]{$\mathcal{Q}^-$}
\obj(50,10)[bbb]{$\tilde{f_-}(\mathcal{Q}^-)$}
\obj(30,0)[ccc]{$\HH\setminus\R$}
\obj(50,0)[ddd]{$f(\HH\setminus\R)$}
\mor{aaa}{bbb}{$\tilde{f_-}$}
\mor{aaa}{ccc}{$\pi$}[-1,0]
\mor{bbb}{ddd}{$\pi$}
\mor{ccc}{ddd}{$f$}
\enddc
$$

On the contrary, if one considers a slice regular function defined on the whole $\HH$, there is no such an issue (see~\cite[Section 7]{gensalsto}).
The aim of what follows is to overcome this problem and to unify the theory.

\begin{prop}
Let $f:\HH\setminus\R\to\HH$ be a slice regular function that splits on $\C_i$ as in Formula~\eqref{split2}. Suppose that the lift $\tilde{f}_+:\mathcal{Q}^+\to\mathbb{CP}^3$
given by the parameterisation in~\eqref{sollevamento+} can be extended to $\mathcal{Q}$ as $\mathcal{F}:\mathcal{Q}\to\mathbb{CP}^3$ by letting varying the variable $v\in\mathbb{CP}^1$. 
Then $\mathcal{F}|_{\mathcal{Q}^-}=\tilde{f_{-}^\vee}$, where $f^\vee:\HH\setminus\R\to\HH$ is the companion of $f$.
\end{prop}

\begin{proof}
Let $f:\HH\setminus\R\to\HH$ be a slice regular function as in the hypothesis and define $\mathcal{F}:\mathcal{Q}\to\mathbb{CP}^3$ to be the 
extension of its lift on $\mathcal{Q}^+$ parameterised as
$$
\mathcal{F}[1,u,v,uv]=[1,u,g(v)-u\hat{h}(v),h(v)+u\hat{g}(v)].
$$
The possibility to extend $\tilde{f}_+$, in particular implies that $g,\hat{g}, h$ and $\hat{h}$ can be extended to $\mathbb{C}$.
Consider now a slice regular function $\varphi$ that splits on $\C_i$ as 
$$
\varphi:=\begin{cases}
l(v)+m(v)j & v\in\C_{i}^{+}\\
\overline{\hat{l}(\bar v)}+\overline{\hat{m}(\bar v)}j & v\in \C_{i}^{-}.
\end{cases}
$$
Then, according to Lemma~\ref{lemmaqminus}, the lift of $\varphi$ on $\mathcal{Q}_-$ is given by
$$
\tilde{\varphi}_-[1,u,v,uv]=[1,u,\overline{\hat{l}(v)}-u\overline{m(v)},\overline{\hat{m}(v)}+u\overline{l(v)}].
$$
For $\tilde{\varphi}_-$ to agree with $\mathcal{F}|_{\mathcal{Q}^-}$ we then have that, for any $v\in\C^-$
$$
\overline{\hat{l}(v)}=g(v),\quad \overline{m(v)}=\hat{h}(v),\quad \overline{\hat{m}(v)}=h(v),\quad \overline{l(v)}=\hat{g}(v),
$$
therefore, the splitting of $\varphi$ is given by
\begin{equation*}
\varphi:=\begin{cases}
\overline{\hat{g}(\bar v)}+\overline{\hat{h}(\bar v)}j & v\in\C_{i}^{+}\\
g(v)+h(v)j  & v\in \C_{i}^{-},
\end{cases}
\end{equation*}
that is $\varphi=f^\vee$, and we have proved the proposition.
\end{proof}

\subsection{Blow-up of the real line}

We now want to give a better description of the role of the real line in this framework.
Since $\HH\setminus\R\simeq \SF\times \C^{+}$, we consider the following embedding
\begin{eqnarray*}
\HH\setminus\R & \hookrightarrow & \SF\times \overline{\C^{+}},\\
\alpha+I\beta & \mapsto & (I,\alpha+i\beta).
\end{eqnarray*}
As recalled at the beginning of this section, the complex manifold $(\HH\setminus\R,\mathbb{J})$ is biholomorphic
to the two open subsets $\mathcal{Q}^{+}$ and $\mathcal{Q}^{-}$ endowed with the standard complex structure.
Computing the pre-images of the real line by means of the projection $\pi$, one has that $\mathcal{Q}_\R:=
\pi^{-1}(\R)\cap\mathcal{Q}\simeq \mathbb{CP}^1\times \mathbb{S}^1$.
Clearly $\overline{\mathcal{Q}^{+}}=\mathcal{Q}^+\cup\mathcal{Q}_\R$ and $\overline{\mathcal{Q}^{-}}=\mathcal{Q}^-\cup\mathcal{Q}_\R$.
As explained in~\cite[Section 7.3]{salamonviac}, $\mathcal{Q}_\R$ is a 3-real-dimensional submanifold
of $\mathcal{Q}$, which disconnects $\mathcal{Q}$ into two components, namely $\mathcal{Q}^+$ and $\mathcal{Q}^-$.
Taking care of these observations, we now extend the twistor projection $\pi|_{\mathcal{Q}}$ to $ \SF\times \overline{\C^{+}}$ as follows. 
If $\alpha\in\R$ and $[s,u]\in\mathbb{CP}^{1}$ corresponds to the element $I\in\SF$, then
$\pi[s,u,s\alpha,u\alpha]=(I,\alpha)=\alpha_I$, where we use the same symbol $\pi$ since there will not be any confusion.

\begin{equation*}
\begin{array}{rccl}
\pi:&\overline{\mathcal{Q}^{+}}&\rightarrow&\SF\times \overline{\C^{+}}\\
  &[s,u,sv,uv]&\mapsto&\begin{cases}
\alpha+I\beta & \mbox{ if } v\in \C^{+}\\
\alpha_{I} & \mbox{ if } v\in\R
\end{cases}
\end{array}
\end{equation*}
The map $\pi$ just defined, obviously extends the twistor projection and is bijective everywhere. For this reason
we may exploit it for our purposes.

Given a slice regular function defined away from the reals, such that the functions defining its splitting can be extended to $\R$,
we want to give an explicit parameterisation of its lift to $\pi^{-1}(\R)\cap\mathcal{Q}$, by means of the just defined extension of $\pi$ to $\SF\times \overline{\C^{+}}$.

\begin{lemma}\label{lemmaqr}
Let $f:\HH\setminus\R\to\HH$ be a slice regular function. Assume that $f$ splits on $\C_i$ as in Formula~\ref{split2} and
that $g,\hat{g}, h, \hat{h}$ can be extended to $\R$. Then defining $f_\R:\SF\times\R\to\HH$ as in Equation~\ref{pr},
there exists a differentiable function $\tilde{f_\R}:\mathcal{Q}_\R\to\mathbb{CP}^3$ such that $f_\R=\pi\circ\tilde{f_\R}\circ\pi^{-1}$.
Moreover, for any $\alpha\in\R$ and for any $I\in\SF$ (corresponding to $[1,u]\in\mathbb{CP}^1$), the parameterisation of $\tilde{f_\R}$ is given by
$$
\tilde{f}_\R[1,u,\alpha,u\alpha]=[1,u,g(\alpha)-u\hat{h}(\alpha),h(\alpha)+u\hat{g}(\alpha)].
$$
\end{lemma}

\begin{proof}
As in the proof of Lemma~\ref{lemmaqminus}, let $F_1+\imath F_2$ be the stem function inducing $f$, and let $q=\alpha+I\beta\in\Omega$, $v=\alpha+i\beta\in\Omega\cap\C_i$ and $Q_u=1+uj$ be such that $\alpha+I\beta=Q_u^{-1} v Q_u$. 
Then, taking into account the definitions of $f_\R$ and of the extension of $\pi$ to $\overline{\mathcal{Q}^{+}}$ and the fact that $\pi$ is bijective, the thesis is a consequence of the following
sequence of equalities. 
\begin{equation*}
  \begin{array}{rcl}
   [1,f(\alpha_{I})] & = & [1,\lim_{\beta\to0^{+}}f(Q_u^{-1}(\alpha+i\beta)Q_u)]\\
   & = & [1,F_1(\alpha)+Q_u^{-1}(i)Q_uF_2(\alpha)]\\
   & = & [Q_u,Q_uF_1(\alpha)+iQ_uF_2(\alpha)]\\
   & = & [1+uj,(1+uj)F_1(\alpha)+i(1+uj)F_2(\alpha)]\\
   & = & [1+uj,F_{1}(\alpha)+iF_{2}(\alpha)+uj(F_{1}(\alpha)-iF_{2}(\alpha))]\\
   & = & [1+uj,f(\alpha_{i})+uj(f(\alpha_{-i}))]\\
   & = & [1+uj,g(\alpha)+h(\alpha)j+uj(\overline{\hat{g}(\alpha)}+\overline{\hat{h}(\alpha)}j)]\\
   & = & \pi[1,u,g(\alpha)-u\hat{h}(\alpha),h(\alpha)+u\hat{g}(\alpha)].
  \end{array}
 \end{equation*}
\end{proof}

\begin{remark}
What we have obtained in the previous lemma can be synthesised by saying that, given a slice regular function $f:\HH\setminus\R\to\HH$ 
that splits as in Formula~\ref{split2} and such that the functions $g,\hat{g},h$ and $\hat{h}$ can be extended to $\R$, then the following diagram 
commute

$$
\begindc{\commdiag}[50]
\obj(-10,10)[aa]{$\overline{\mathcal{Q}^+}$}
\obj(10,10)[bb]{$\tilde{f_+}(\overline{\mathcal{Q}^+})$}
\obj(-10,0)[cc]{$\SF\times \overline{\C^{+}}$}
\obj(10,0)[dd]{$f(\SF\times \overline{\C^{+}})$}
\mor{aa}{bb}{$\tilde{f_+}$}
\mor{aa}{cc}{$\pi$}[-1,0]
\mor{bb}{dd}{$\pi$}
\mor{cc}{dd}{$f$}
\obj(30,10)[aaa]{$[1,u,v,uv]$}
\obj(60,10)[bbb]{$[1,u,g(v)-u\hat{h}(v),h(v)+u\hat{g}(v)]$}
\obj(30,0)[ccc]{$\alpha+I\beta$}
\obj(60,0)[ddd]{$f(\alpha+I\beta)$}
\mor{aaa}{bbb}{$\tilde{f_+}$}
\mor{aaa}{ccc}{$\pi$}[-1,0]
\mor{bbb}{ddd}{$\pi$}
\mor{ccc}{ddd}{$f$}
\enddc
$$
Analogous considerations hold for $\tilde f_{-}$.

\end{remark}

%

\section{Pre-images of a slice-polynomial function}

Let $P$ be a slice-polynomial function which splits on $\C_i$ as in Formula~\eqref{standardsplitting}.
We denote by $\tilde{P}_+:\mathcal{Q}^+\to \mathbb{CP}^3$ the lift of $P$ on $\mathcal{Q}^+$, by $\tilde{P_{-}^\vee}:\mathcal{Q}^-\to \mathbb{CP}^3$ the lift of $P^\vee$ on 
$\mathcal{Q}^-$ and by $\tilde{P_{\R}}:\mathcal{Q}_\R\to \mathbb{CP}^3$ the lift of $P_{\R}$ on 
$\mathcal{Q}_\R$. 
\begin{defi}
We define the {\em extended twistor lift} of $P$ as the function	$\mathcal{P}:\mathcal{Q}\rightarrow \mathbb{CP}^3$ given by
\begin{equation*}
 \mathcal{P}(X):=\begin{cases}
                \tilde{P}_+(X), &\mbox{ if }X\in\mathcal{Q}^+,\\
                 \tilde{P_{\R}}(X), &\mbox{ if }X\in\mathcal{Q}_\R ,\\
                \tilde{P_{-}^\vee}(X), &\mbox{ if }X\in\mathcal{Q}^-.
                \end{cases}
\end{equation*}
\end{defi}

\noindent Thanks to~\cite[Theorem 5.3]{gensalsto} and its extensions (the already cited~\cite[Theorem 24]{AAtwistor}
and Lemmas~\ref{lemmaqminus} and~\ref{lemmaqr}), we are able to give a parameterisation of $\mathcal{P}$
by means of the splitting of $P$: 
\begin{equation*}
\mathcal{P}[1,u,v,uv]=[1,u,g(v)-u\hat{h}(v),h(v)+u\hat{g}(v)],
\end{equation*}
where $u,v\in\C_i$. 
By its definition $u$ parameterises all the imaginary units of
$\SF$ except for $-i$. For this reason, recalling the \textit{Segre embedding} $(\mathbb{CP}^1\times\mathbb{CP}^1)\simeq\mathcal{Q}$ given by $([s,u],[t,v])\mapsto[st,ut,sv,uv]$,
it is useful to consider the parameterisation of $\mathcal P$ homogenised 
with respect to the variable $u$
$$
\mathcal{P}[s,u,sv,uv]=[s,u,sg(v)-u\hat{h}(v),sh(v)+u\hat{g}(v)].
$$

At this stage, the idea is to study simultaneously the pre-images of a point $q\in \HH$ via $P$, $P^{\vee}$ and $P_{\R}$ by means of their extended twistor lift $\mathcal{P}$. This is possible having in mind the commutative properties of the following three diagrams
$$
\begindc{\commdiag}[50]
\obj(-10,10)[aa]{$\mathcal{Q}^+$}
\obj(10,10)[bb]{$\tilde{P_+}(\mathcal{Q}^+)$}
\obj(-10,0)[cc]{$\HH\setminus\R$}
\obj(10,0)[dd]{$P(\HH\setminus\R)$}
\mor{aa}{bb}{$\tilde{P_+}$}
\mor{aa}{cc}{$\pi$}[-1,0]
\mor{bb}{dd}{$\pi$}
\mor{cc}{dd}{$P$}
\enddc\quad
\begindc{\commdiag}[50]
\obj(-10,10)[aa]{$\mathcal{Q}^-$}
\obj(10,10)[bb]{$\tilde{P^{\vee}_-}(\mathcal{Q}^-)$}
\obj(-10,0)[cc]{$\HH\setminus\R$}
\obj(10,0)[dd]{$P^{\vee}(\HH\setminus\R)$}
\mor{aa}{bb}{$\tilde{P^{\vee}_-}$}
\mor{aa}{cc}{$\pi$}[-1,0]
\mor{bb}{dd}{$\pi$}
\mor{cc}{dd}{$P^{\vee}$}
\enddc\quad
\begindc{\commdiag}[50]
\obj(-10,10)[aa]{$\mathcal{Q}_\R$}
\obj(10,10)[bb]{$\tilde{P_{\R}}(\mathcal{Q}_\R)$}
\obj(-10,0)[cc]{$\SF\times \R$}
\obj(10,0)[dd]{$P_{\R}(\SF\times \R)$}
\mor{aa}{bb}{$\tilde{P_{\R}}$}
\mor{aa}{cc}{$\pi$}[-1,0]
\mor{bb}{dd}{$\pi$}
\mor{cc}{dd}{$P_{\R}$}
\enddc
$$
and the fact that $\mathcal{Q}_\R$ disconnects $\mathcal{Q}$ into $\mathcal{Q}^{+}$ and $\mathcal{Q}^{-}$.
Let us denote by $\mathcal{S}=\mathcal{P}(\mathcal{Q})\subset\mathbb{CP}^3$  the surface parameterised by the extended lift of $P$. If a point $q\in\HH$ lifted to $\mathcal{S}=\tilde{P_+}(\mathcal{Q}^+)\cup\tilde{P^{\vee}_-}(\mathcal{Q}^-)\cup\tilde{P_{\R}}(\mathcal{Q}_\R)$ has a pre-image via $\mathcal{P}$
in $\mathcal{Q}^{+}$, $\mathcal{Q}^{-}$ or $\mathcal{Q}_\R$, then it will have 
a pre-image in $\mathbb{S}\times\overline{\mathbb{C}^+}$ with respect to $P$, $P^{\vee}$ or $P_{\R}$,
respectively.

 Since the fibres of $\pi$ are parameterised by System~\eqref{eqfiber}, then, for any $q=q_{1}+q_{2}j\in\mathbb{H}$, 
we have that the intersection $\pi^{-1}(q)\cap\mathcal{S}$ is parameterised by 

\begin{equation*}
\begin{cases}
sg(v)-u\hat{h}(v)=sq_{1}-u\bar q_{2}\\
sh(v)+u\hat{g}(v)=sq_{2}+u\bar q_{1}
\end{cases}
\end{equation*}
%
and, therefore, we are left to find solutions $([s,u],v)\in \mathbb{CP}^1\times \C$ of the system
\begin{equation}\label{system}
\begin{cases}
s(g(v)-q_1)-u(\hat h (v)-\bar q_2)=0\\
s(h(v)-q_2)-u(\bar q_1-\hat g (v))=0.
\end{cases}
\end{equation}
Let us denote by $A_q(v)$ the two by two complex matrix associate to the previous system, 
\begin{equation*}
	A_q(v)=
{\begin{pmatrix}
	 g(v)-q_1 & -\hat h (v)+\bar q_2\\
	h(v)-q_2 & -\bar q_1+\hat g (v)
\end{pmatrix},
}\end{equation*}
 by 
$D_q(v)$ the determinant of $A_q(v)$,
\begin{equation}\label{determinantA}
D_q(v)=-(g(v)-q_1)(\bar q_1-\hat g (v))+(h(v)-q_2)(\hat h (v)-\bar q_2),
\end{equation}
and by $d=\max_{q\in\HH}\{\deg D_q(v)\}$ its degree. This complex polynomial depends on $q$ but its degree is generically constant and
equal to $d$.
The following theorem describes the set of solutions of System~\eqref{system} for a fixed $q$.

\begin{theorem}\label{systempre-images}
Let $P$ be a slice-polynomial function, not slice-constant, which splits on $\C_i^+$ as
\begin{equation*}
P(v):=\begin{cases}
g(v)+h(v)j & v\in\C_{i}^{+}\\
\overline{\hat{g}(\bar v)}+\overline{\hat{h}(\bar v)}j & v\in \C_{i}^{-},
\end{cases}
\end{equation*}
let $q=q_1+q_2j \in \HH$ be any quaternion and let $A_q(v)$, $D_q(v)$ and $d$ be defined as above. Then
\begin{itemize}
\item if there exists $v_{0}\in\mathbb{C}_{i}$ such that $A_q(v_{0})=0$, then $([s,u],v_0)$ is a solution of System \eqref{system} for any $[s,u]\in \mathbb{CP}^1$;

		\item if  $D_q(v)\equiv 0$ then, for any $v$ such that $A_q(v)\neq 0$,  there exists and is unique $[s(v),u(v)]\in \mathbb{CP}^1$ such that $([s(v),u(v)],v)$ is a solution of System \eqref{system};

			\item if  $D_q(v)\not \equiv 0$ and $v_1,\ldots, v_d$ are the roots of $D_q(v)$  repeated according to their multiplicity, then, if $A_q(v_{k})\neq 0$, there exists and is unique $[s_{k},u_{k}]\in \mathbb{CP}^1$ such that $([s_{k},u_{k}],v_k)$ is a solution of System \eqref{system}.
%
%
%
%
%
%
%
\end{itemize}


\end{theorem}
\begin{proof}
Note that the fact that $P$ is not slice-constant implies that $A_q(v)$ is not a constant matrix.	
In order to get solutions $[s,u] \in \mathbb{CP}^1$ of System \eqref{system}, the determinant $D_q(v)$ must vanish.
If there exists $v_{0}$ such that $A_q(v_{0})=0$, then it is easy to see that for any $[s,u]\in \mathbb{CP}^1$ the pair $([s,u],v_0)$ is a solution of System \eqref{system} in $\mathcal{Q}$.

Suppose now that $D_q(v)$ vanishes identically. This is equivalent to say that the rows {$(A_q)_1(v)=( g(v)-q_1,  -\hat h (v)+\bar q_2)$ and $(A_q)_2(v)=(h(v)-q_2,  -\bar q_1+\hat g (v))$} of $A_q(v)$ are identically linearly dependent. 
{  Suppose then that, without loss of generality, $(A_q)_1(v)\not \equiv (0,0)$. Then the unique solution of System \eqref{system} in $\mathbb{CP}^1\times \C$ is given by $([s(v),u(v)],v)=([\hat h (v)-\bar q_2, g(v)-q_1],v)$.
}

If otherwise $D_q(v)$ is not identically zero, then it has exactly $d$ roots $v_1, \ldots, v_d$ repeated according to their multiplicity. 
If $A_{q}(v_{k})\neq 0$ (supposing without loss of generality that $(A_q)_1(v_k)\neq (0,0)$) the pair $([\hat h (v_k)-\bar q_2, g(v_k)-q_1],v_k)$ is   a solution of System \eqref{system} in $\mathcal{Q}$.
\end{proof}
The previous result can be interpreted in terms of pre-images of slice-polynomial functions. In fact, in the following corollary, we give a complete description
of the simultaneous pre-image of a point $q\in\HH$ via a slice-polynomial function, its companion and its extension to the reals.

\begin{cor}\label{mainresult}
Let $P$ be a slice-polynomial function, not slice-constant, which splits on $\C_i^+$ as in Formula \eqref{standardsplitting}, 
	let $q$ be any quaternion and let $A_{q}(v)$, $D_q(v)$ and $d$ be defined as above. Then

	\begin{itemize}
	\item if there exists $v_{0}\in\mathbb{C}_{i}$ such that $A_q(v_{0})=0$, then 
			if $v_{0}\notin \R$, then $P(\SF_{v_0})\equiv P^\vee(\SF_{v_0}) \equiv q$, 
			otherwise, if $v_{0}\in\R$, then $P_{\R}(v_{0})=q$ and $P, P^{\vee}$ can be extended
			continuously to $v_{0}$ as $P(v_{0})=P^{\vee}(v_{0})=q$;

		\item if  $D_q(v)\equiv 0$ then there exist two real surfaces $ \Sigma, \Sigma^\vee\subset\HH\setminus\R$, both biholomorphic to $\C^+$, such that $P(\Sigma)\equiv P^\vee(\Sigma^\vee) \equiv q$;

			\item if  $D_q(v)\not \equiv 0$, let $v_1,\ldots, v_d$ be the roots of $D_q(v)$  repeated according to their multiplicity. If $A_q(v_{k})\neq 0$ for all $k$,  then \[\# \{P^{-1}(q)\}+\# \{{P^\vee}^{-1}(q)\}+\# \{P_\R^{-1}(q)\}=d.\] 
		
		
	\end{itemize}

\end{cor}

\begin{proof}
Suppose first that there exists $q\in\mathbb\HH$ and $v_{0}\in\mathbb{\C}$ such that $A_q(v_{0})= 0$. As shown
in Theorem~\ref{systempre-images}, for any $[s,u]\in\mathbb{CP}^{1}$, $([s,u],v_{0})$ is a solution 
of the System~\eqref{system}. If $v_{0}\in\C\setminus\R$, since $\pi:\mathcal{Q}^{+}\to\HH\setminus\R$
and $\pi:\mathcal{Q}^{-}\to\HH\setminus\R$ are both biholomorphisms, thanks to Proposition~\ref{degeneratecoincide}, we have that
$\mathbb{S}_{v_{0}}=P^{-1}(q)={P^{\vee}}^{-1}(q)$. If $v_{0}\in\R$, then, thanks to Remark~\ref{parterealedegenere} and to Lemma~\ref{lemmaqr}, $\mathbb{S}\times \{v_{0}\}=P_{\R}^{-1}(q)$,
that is $P_{\R}(v_{0})=q$.

If now $D_q(v)\equiv 0$, then, for any $v\in\C$ such that $A_q(v)\neq 0$, there exists $[s(v),u(v)]\in\mathbb{CP}^{1}$, such that $([s(v),u(v)],v)$ is a solution of System \eqref{system}. Suppose for simplicity that the first line $(A_{q})_{1}(v)$ is not identically zero,
then, $([\tilde{h}(v)-\bar q_{2},g(v)-q_{1}],v)\subset\mathbb{CP}^{1}\times\C$ defines a holomorphic 
surface $\mathcal{S}$. 
If for some $v_{0}$ $A_{q}(v_{0})=0$, then thanks Theorem~\ref{systempre-images}, for all $[s,u]\in\mathbb{CP}^{1}$, $([s,u],v_{0})$ is a solution of the System~\eqref{system}, and so, in particular
the surface $([\tilde{h}(v)-\bar q_{2},g(v)-q_{1}],v)\subset\mathbb{CP}^{1}\times\C$ can be extended
to $v_{0}$ holomorphically by taking $([\tilde{h}(v_{0})-\bar q_{2},g(v_{0})-q_{1}],v_{0})$.
Now, since, again $\pi$ is a biholomorphism if restricted to $\mathcal{Q}^{+}$
or to $\mathcal{Q}^{-}$, then $\Sigma:=\pi(\mathcal{S}\cap\mathcal{Q}^{+})\subseteq P^{-1}(q)$ and 
$\Sigma^{\vee}:=\pi(\mathcal{S}\cap\mathcal{Q}^{-})\subseteq {P^{\vee}}^{-1}(q)$ are two holomorphic surfaces. 

In the remaining case, let $v_{1},\dots, v_{d}$ be the roots of $D_q(v)$ repeated according to their multiplicity.
Suppose that $A_{q}(v_{k})\neq 0$ for all $k$, then by Theorem~\ref{systempre-images}, there exists and is unique $[s_{k},u_{k}]\in\mathbb{CP}^{1}$ such that $\pi(\mathcal{P}[s_{k},u_{k},s_{k}v_{k},u_{k}v_{k}])=q$. But then,
depending whether $v_{k}$ belongs to $\C^{+}, \C^{-}$ or to $\R$, it will provide a pre-image
for $P, P^{\vee}$ or $P_{\R}$, respectively and these three options are obviously mutually exclusive.
\end{proof}
Thanks to the previous corollary we are able to give the following definition.

\begin{defi}\label{twdeg}
Let $P$ be a slice-polynomial function. We define the \textit{twistor degree} of $P$ as $\ddeg(P)=d=\deg D_q$,
where $D_q$ is defined as in Formula~\eqref{determinantA}.
\end{defi}


The two surfaces $\Sigma$ and $\Sigma^\vee$ in the statement of Corollary~\ref{mainresult} glue together at
the real line, in the following sense: the surface $\mathcal S$ defined in the proof of the Corollary parameterises
a holomorphic surface in $\mathbb{CP}^1\times\C\subset\mathcal Q$. But $\pi:\mathcal Q \to \HH$ is a 
class $\mathcal{C}^\infty$ map whose restrictions to $\mathcal{Q}^+$ and $\mathcal{Q}^-$ have $\Sigma$ and
$\Sigma^\vee$ as images, respectively and these, of course, glue together at their boundary, given by $\pi(\mathcal{S}\cap\mathcal{Q}_\R)$.

\begin{remark}\label{doubledegree}
Given a quaternionic polynomial $Q$ of degree $n$ we have that $\overline{\hat{g}(\bar v)}= g(v)$ and $\overline{\hat{h}(\bar v)}= h( v)$, hence $\ddeg(Q)=2n$. Moreover, thanks to these symmetries, to 
any solution $v\in\C_{i}^{+}$ of $D_{q}(v)=0$ corresponds another solution in $\C_{i}^{-}$, coherently 
with the fact that $Q^{\vee}=Q$. In particular, for this reason, it was not possible to see all this peculiar
aspects of the theory in the context of slice regular functions defined over the reals as in the case
studied in~\cite{gensalsto}.
\end{remark}

\begin{remark}\label{systemsliceconstant}
	If $P$ is a slice-constant function, then $g,\hat g,h,\hat h$ are constant functions and then System \eqref{system} reduces to
	\begin{equation*}
	\begin{cases}
	sA-uB=0\\
	sC-uD=0.
	\end{cases}
	\end{equation*}
fore some constants $A,B,C,D \in \C_i$. Hence the system has admissible solutions if and only if the corresponding determinant $D_q(v)=-AD+BC= 0$. In this case for any $v \in \C_i$ the unique solution is given by $[s,u]=[B,-A] \in \mathbb{CP}^1$.
\end{remark}

\begin{example}\label{esempicompleti}
Let us exhibit some examples of analysis of the System in~\ref{system}.
\begin{itemize}
\item Consider the slice-constant function $\ell_+$. For this function the splitting is given by $g\equiv 1$
and $\hat{g}\equiv h\equiv \hat {h}\equiv 0$. Therefore, for any $v\in \C_i$, the matrix $A_q$ is given by
\begin{equation*}
	A_q(v)=
\begin{pmatrix}
	 1-q_1 & -\bar q_2\\
	-q_2 & \bar q_1
\end{pmatrix},
\end{equation*}
and $D_q(v)=\bar q_1-|q_1|^2-|q_2|^2$.  This is a constant function and is equal to zero only when $\bar q_1\in\R$
and $q_1=|q|^2$. If $q=\alpha+I\beta$, then, translating the last condition, the image of $\ell_+$ can be expressed
as follows
$$
\ell_+(\HH\setminus\R):=\{\alpha+I\beta\in\HH\,|\, I\perp i,\, \alpha=\alpha^2+\beta^2\},
$$
representing, for any $I\in\SF$ a circle of radius $1/2$ passing through $0$ and $1$ with center in $1/2$.
As expressed in Remark~\ref{systemsliceconstant}, for any $q\in\ell_+(\HH\setminus\R)$ there is an element
$[s,u]\in\mathbb{CP}^1$ which solves System~\eqref{system} for any $v\in\C$. The last implies that
any $q\in\ell_+(\HH\setminus\R)$ has a semi-slice of pre-images.
\item We pass now to the function $q\ell_+$. This function splits as $g(v)=v$
and $\hat{g}\equiv h\equiv \hat {h}\equiv 0$. The matrix $A_q$ is given by
\begin{equation*}
	A_q(v)=
\begin{pmatrix}
	 v-q_1 & -\bar q_2\\
	-q_2 & \bar q_1
\end{pmatrix},
\end{equation*}
and $D_q(v)=v\bar q_1-|q_1|^2-|q_2|^2$. This is identically zero for $q=0$ and for this value, the system in Formula~\eqref{system} as a solution for $s=0$, showing that $q\ell_+|_{\C_i^+}\equiv 0$. For any $q\neq 0$ there exists and is 
unique the solution $([s,u],v)$ of System~\eqref{system}.
Notice moreover that, for $q=0$ it holds $A_{0}(0)\equiv 0$, therefore $q\ell_{+}$ can be defined in $0$
and $0$ is degenerate for the function.
\item We study now the slice-polynomial function $(q+j)*\ell_+=\ell_+v+\ell_-j$.
The matrix $A_q(v)$ for this function is the following
\begin{equation*}
	A_q(v)=
\begin{pmatrix}
	 v-q_1 & 1-\bar q_2\\
	-q_2 & \bar q_1
\end{pmatrix},
\end{equation*}
and $D_q(v)=v\bar q_1+q_2-|q_1|^2-|q_2|^2$. This polynomial is again identically zero for $q=0$ but also for $q=j$.
For these two values for any $v\in\C$ there exists $[s,u]\in\mathbb{CP}^1$ that solve the system in~\eqref{system}.
In particular, for any $v$, the point $q=0$ is obtained for $[1,v]\in\mathbb{CP}^1$, while $q=j$, for $[0,1]$. 
For this function $A_{q}(v)\neq 0$ for any $q\in\HH$.
\end{itemize}
\end{example}

	\begin{cor}\label{union}
Let $P$ be a slice-polynomial function, not slice-constant. Then
\[P(\HH\setminus \R)\cup P^\vee(\HH\setminus \R)\cup P(\R)=\HH.\]
	\end{cor}
\begin{proof}
It is a direct consequence of Theorem~\ref{systempre-images} and Corollary~\ref{mainresult}, keeping in mind that, if $P$ is not slice-constant, 
then at least one among $g, \hat{g}, h$ and $\hat{h}$ is not a constant polynomial. Therefore, for any $q\in\HH$
$D_q(v)=0$ has at least a root and the System~\eqref{system} can be solved.
\end{proof}

Since the lift $\mathcal{P}$ of a slice-polynomial function is parameterised by polynomial functions, then its image lies
on an algebraic surface $\mathcal{K}$. The twistor projection has $\mathbb{CP}^1$ as fibres and
therefore, they generically intersect $\mathcal{K}$ in $\deg\mathcal{K}$ points.
Hence we have the following corollary.

\begin{cor}\label{surf}
Let $P$ be a slice-polynomial function, which splits on $\C_i$ as in Formula~\eqref{standardsplitting}
and let $d=\ddeg(Q)$ be its twistor degree. 
Assume that the extended twistor lift $\mathcal{P}$ of $P$ is generically $k:1$ on its image, then
 the image of $\mathcal{P}$ lies in an algebraic surface of degree $d/k$.
\end{cor}

\begin{proof}
The thesis is trivial and we just add the following synthesising graph.
$$
\begindc{\commdiag}[50]
\obj(-10,10)[aa]{$\mathcal{Q}$}
\obj(10,10)[bb]{$\mathcal{K}$}
\obj(10,0)[dd]{$\HH$}
\mor{aa}{bb}{$\mathcal{P}$}
\mor{aa}{bb}{$k:1$}[-1,0]
\mor{aa}{dd}{$d:1$}[-1,0]
\mor{bb}{dd}{$\pi$}
\mor{bb}{dd}{$\frac{d}{k}:1$}[-1,0]
\enddc\quad
$$
\end{proof}

\begin{remark}
The last corollary clarifies some issues of the twistor interpretation of slice regularity. First of all, thanks to 
Remark~\ref{doubledegree}, given any quaternionic polynomial $Q$ of degree $n$, one has that 
$d=\ddeg(Q)=2n$. Therefore, if the extended lift $\mathcal{Q}$ of $Q$ is generically $1:1$ and it has image on an algebraic surface $\mathcal{K}$, then $\deg\mathcal{K}=d=2n$. Hence, 
given an algebraic surface $\mathcal{K}\subset \mathbb{CP}^{3}$ of odd degree, it is not possible
to parameterise it by means of the twistor lift of a slice regular function in a generically injective way
(see e.g.~\cite[Sections 6 and 7]{gensalsto} where a quaternionic polynomial of degree 2 allows
to study a rational quartic scroll).
Moreover, the content of the present section completely solves the problem exposed
 in~\cite[Remark 17]{AAtwistor} where it was pointed out that it is not sufficient to consider the twistor lift of 
 a slice regular function to parameterise even simple algebraic surfaces. In particular, the \textit{ghost function} 
 proposed in~\cite[Remark 17]{AAtwistor} is what we call {\em companion}.
\end{remark}

\begin{remark}
Most of the techniques introduced in this section can be applied to a generic
slice regular function $f:\HH\setminus\R\to\HH$ for the study of its pre-images obtaining general results.
For instance the Open Mapping Theorem does not imply that the image of a slice regular function defined on
a domain without real points is an open set (see~\cite[Remark 5.2]{AAproperties}). However
 Corollary~\ref{union} shows that the union of the three images $P(\HH\setminus \R)\cup P^\vee(\HH\setminus \R)\cup P(\R)$ is an open set. {This property is likely to be generalised to the case of slice regular functions.
This study goes beyond the aims of the present work but it would be really interesting to explore it in the future.}
\end{remark}

\section{Discriminant locus of a cubic scroll}
The aim of this section is to describe the discriminant locus of a cubic scroll in $\mathbb{CP}^3$ by means of the introduced techniques
on slice-polynomial functions.
In~\cite{AAtwistor} it is proved that the family of surfaces that can be parameterised by the twistor lift of a slice regular function consists
only of scrolls (i.e.: surfaces ruled by lines). Let us study, as a meaningful example, the cubic scroll $\mathcal{C}\subset\mathbb{CP}^3$ defined by the set of points $[X_0,X_1,X_2,X_3]\in \mathbb{CP}^3$ satisfying the equation
\begin{equation}\label{scroll}
X_0X_3^2+X_1^2X_2=0.
\end{equation}
In the Cayley classification~\cite{AC1, AC2, AC3}, the cubic $\mathcal{C}$ is denoted by $S(1,1)$ (or by $S(1,1,3)$). 
Using the notation $m_{ij}:=\{[X_0,X_1,X_2,X_3]\,|\,X_i=0=X_j\}$, we have that $\mathcal{C}$ is a non-normal surface, singular along the line $m_{13}$
and it has two cuspidal points at $A=[1,0,0,0]$ and $B=[0,0,1,0]$. It has two directrix lines, namely $m_{13}$ and $m_{02}$.

In~\cite{AAtwistor} the first author constructs a slice regular function whose twistor lift parameterises $\mathcal{C}$. The mentioned function is one of the slice-polynomial functions introduced in Example~\ref{firstexample}:
$$
R(q)=-q^2\ell_+ +q\ell_-=\ell_+(-v^2)+\ell_-(\bar v).
$$
By the second expression of $R$ we immediately see that its splitting of $R$ with respect to $\C_i$ is given by the functions $g(v)=-v^2$, $\hat{g}(v)=v$ and $h=\hat{h}\equiv 0$.
Therefore its twistor degree is $\widetilde{\deg}(R)=3$ and the extended twistor lift $\mathcal{R}:\mathcal{Q}\to\mathbb{CP}^3$ of $R$ is given by
$$
\mathcal{R}[s,u,sv,uv]=[s,u,-sv^2,uv].
$$
Considering the homogenisation with respect to the variable $v$, the map $\mathcal{R}$ si given by:
$$
\mathcal{R}([s,u],[t,v])=\mathcal{R}[st,ut,sv,uv]=[st^2,ut^2,-sv^2,uvt].
$$
In the same spirit of~\cite[Theorem 7.1]{gensalsto} we state the following.
\begin{theorem}\label{birat}
The mapping $\mathcal{R}$ is a birational equivalence between $\mathcal{Q}$ and $\mathcal{C}$.
In fact, $\mathcal{R}$ maps $\mathcal{Q}$ onto $(\mathcal{C}\setminus m_{01})\cup\{B\}$ and is injective on the complement of
$m_{13}\setminus\{A\}$ where it has $2$ pre-images except for the point $B$, where $\mathcal{R}^{-1}(B)\simeq \mathbb{CP}^1$.
\end{theorem}
\begin{proof}
First of all, if $X_0\neq 0$, $X_1\neq 0$, $X_2\neq 0$ and $X_3\neq 0$, we can define
$$
\frac{s}{u}=\frac{X_0}{X_1},\qquad\mbox{and}\qquad\frac{v}{t}=\frac{X_1X_2}{X_0X_3}
$$
and $[X_0,X_1,X_2,X_3]$ has a unique inverse image. The other cases are the following:
\begin{itemize}
\item if $X_0=0$, then $X_1^2X_2=0$, obtaining the lines $m_{01}$ and $m_{02}$.
\item If $X_1=0$, then $X_0X_3^2=0$, obtaining the lines $m_{01}$ and $m_{13}$.
\item If $X_2=0$, then $X_0X_3^2=0$, obtaining the lines $m_{02}$ and $m_{23}$.
\item If $X_3=0$, then $X_1^2X_2=0$, obtaining the lines $m_{13}$ and $m_{23}$.
\end{itemize}
The four lines $m_{01}, m_{02}, m_{13}, m_{23}$ intersect at $A,B, C=[0,0,0,1]$ and $D=[0,1,0,0]$ (see Figure~\ref{fourlines}).
The line $m_{02}\setminus\{C,D\}$ can be obtained for $s=0$ and, in this case the map $\mathcal{R}$
can be inverted by setting 
$$
\frac{v}{t}=\frac{X_1}{X_3}.
$$
The point $D$ is obtained for $s=0=v$ and so it has just a pre-image, while the point $C$ has no pre-images
since $X_0=X_1=0$ forces $t=0$.
The line $m_{13}\setminus\{A,B\}$ can be obtained for $u=0$ and the parameterisation gives two pre-images by 
$$
\left(\frac{v}{t}\right)^2=\frac{X_2}{X_0}.
$$ 
The point $A$ is obtained for $u=0=v$ and so it has only a pre-image, while the point $B$ can
be obtained for $t=0$ and since there are no condition on $[s,u]$, then $\mathcal{R}^{-1}(B)\simeq \mathbb{CP}^1$.
The line $m_{23}\setminus\{A,D\}$ can be obtained for $v=0$ and each point is given by
$$
\frac{u}{s}=\frac{X_1}{X_0}.
$$
Lastly, the line $m_{01}$ would be obtained for $t=0$, but this forces $X_3=0$. Therefore $m_{01}$ is not
in the image of $\mathcal{R}$ except for the point $B$.
\begin{figure}[ht]
\includegraphics[width=0.50\linewidth, height=5cm]{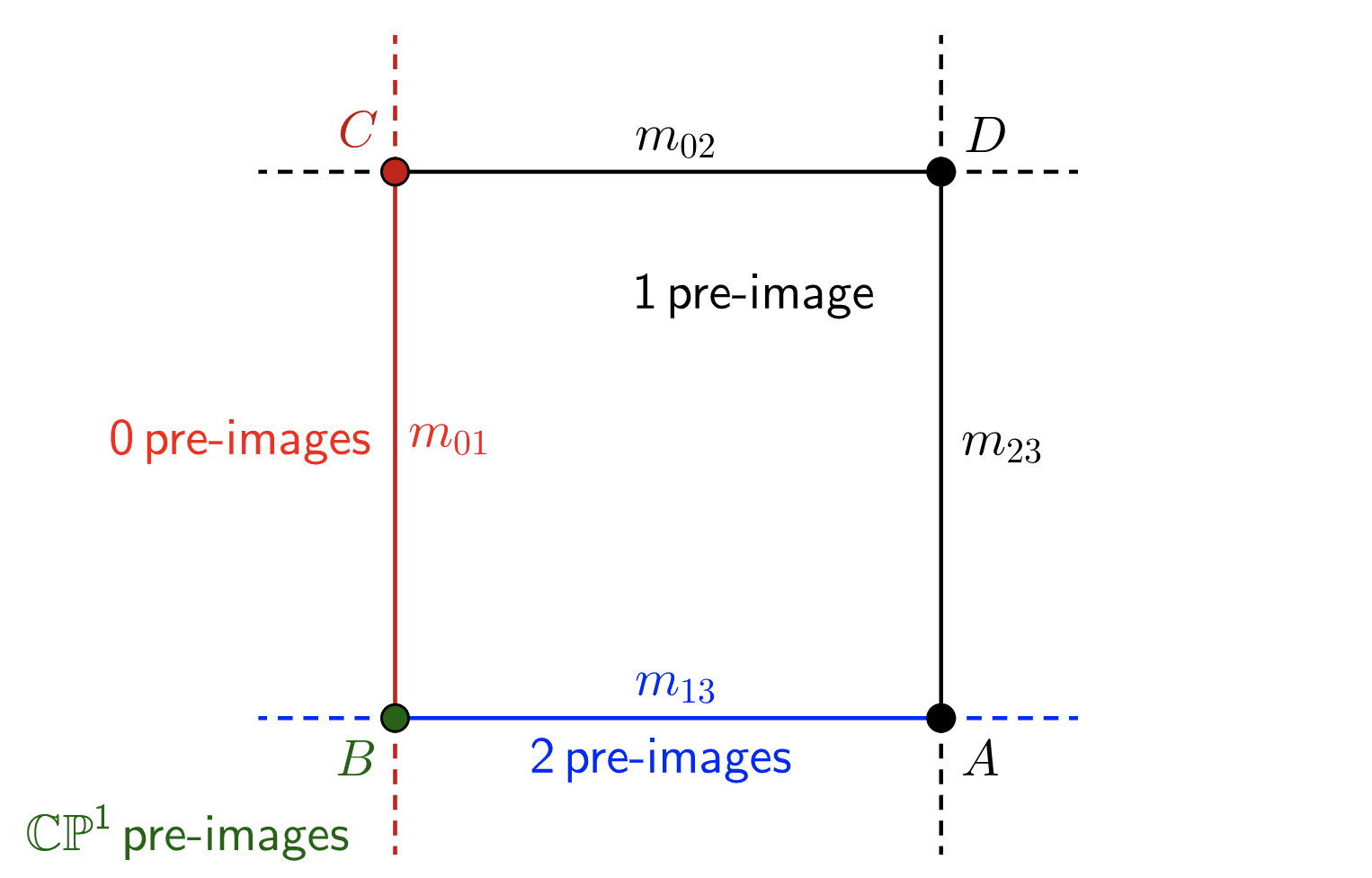}
\caption{In black the set of points with one pre-image, in blue the set of points with 2 pre-images, in red the set of points with no pre-image and in green the point $B$ with a $\mathbb{CP}^1$ of pre-images}
\label{fourlines}
\end{figure}\\
\end{proof}

{Given an algebraic surface  $\mathcal{K}$ in $\mathbb{CP}^3$ of degree $d$, its intersection with a twistor fiber consists in exactely $d$ points except at points belonging to the so called discriminant locus, where it can be either of lesser cardinality or it can be an entire fiber $\mathbb{CP}^1$. More precisely, we recall the following definition.
}

\begin{defi}
Let $\mathcal{K}$ be a surface of degree $d$ in $\mathbb{CP}^3$.
We define the \textit{discriminant locus} of $\mathcal{K}$ as the following set
$$
Disc(\mathcal{K}):=\{q\in\SF^4\simeq\mathbb{HP}^1\,|\, \#(\pi^{-1}(q)\cap\mathcal{K})\neq d\}.
$$
We denote by $G_0$ the subset of $Disc(\mathcal{K})$ such that for any $q\in G_0$, the twistor fibre $\pi^{-1}(q)\subset\mathcal{K}$
and, for any $n<d$ we define the following sets
$$
G_n:=\{q\in\SF^4\simeq\mathbb{HP}^1\,|\, \#(\pi^{-1}(q)\cap\mathcal{K})=n\}.
$$
Hence $Disc(\mathcal{K})=G_0\cup G_1\cup \dots\cup G_{d-1}$.
In the particular case in which $d=3$, we have $Disc(\mathcal{K})=G_0\cup G_1\cup G_2$ and the set $G_1$ is called the set of \textit{triple points} and $G_{2}$ the set of \textit{double points}.
\end{defi}

Now, in view of the interest related to the discriminante locus of an algebraic surface in the twistor space (see~\cite{AAtwistor, armstrong, APS, AS, gensalsto, salamonviac}),
we are going to describe it for the cubic scroll $\mathcal{C}$. To do that we will exploit the new techniques on slice regularity concerning
slice-polynomial functions and their pre-images. Thanks to Theorem~\ref{birat}, we know that the parameterisation $\mathcal{R}$ is  $1:1$ outside
 a Zarisky-closed subset of $\mathcal{C}$, namely $m_{01}\cup m_{13}$. Therefore, outside this set, we can describe the discriminant locus of $\mathcal{C}$ by means of Theorem~\ref{systempre-images}. 
The set $m_{01}\cup m_{13}$ will be studied by its own in the following remark.

\begin{remark}\label{lines}
{Recalling the definition of $\pi$, it is immediate to see that the line $m_{01}$ projects to $\infty$, which entails that $m_{01}$ is a twistor fiber and that $\infty$ belongs to $G_{0}$. 
} 

Consider now the intersection between $m_{13}$ and the twistor fibers in~\eqref{eqfiber}. 
We have that $X_{2}=X_{0}q_{1}$ and $0=X_{0}q_{2}$ implying $q_{2}=0$.
Hence $m_{13}$ projects on $\hat{\C}_{i}$. 
Since $m_{13}\subset\pi^{-1}(\hat{\C}_{i})\cap \mathcal{C}$ and $\mathcal{R}^{-1}$ is not defined on $m_{13}$
 we study here separately this set of fibers.

Consider the intersection between the twistor fibre \eqref{eqfiber} above $[1,q_1+q_2j]\in\HH\mathbb{P}$ and the cubic $\mathcal{C}$, 
\begin{equation}\label{intersection}
X_{0}^{3}q_{2}^{2}+2X_{0}^{2}X_{1}\bar q_{1}q_{2}+X_{0}X_{1}^{2}(\bar q_{1}^{2}+q_{1})-X_{1}^{3}\bar q_{2}=0.
\end{equation}
For any $q\in\HH$ such that $q_{2}=0$, Equation~\eqref{intersection}
becomes $X_{0}X_{1}^{2}(\bar q_{1}^{2}+q_{1})=0$. 
The set of points $q_{1}$ such that $\bar q_{1}^{2}+q_{1}=0$ clearly consists of  elements in $G_{0}$, $q_{1}\in\{0,-1,1/2+\sqrt{3}/2,1/2-\sqrt{3}/2\}\subset G_{0}$.
Therefore, for any $q_{1}$ such that $(\bar q_{1}^{2}+q_{1})\neq 0$, there are two solutions, namely $X_{0}=0$ and $X_{1}=0$ with multiplicity two and so $(\hat{\C}_{i}\setminus G_{0})\subset G_{2}$.
\end{remark}

Having analysed the two lines $m_{01}$ and $m_{13}$, we now pass to look at the discriminant
locus of $\mathcal{C}$ by means of the parameterisation $\mathcal{R}$.

First we recall some fundamental notion on cubic polynomials. Given any complex polynomial of degree $3$,
$p(x)=ax^{3}+bx^{2}+cx+d$, its discriminant is given by the following quantity
$$
\Delta_{p}=18abcd - 4b^3d + b^2c^2 - 4ac^3 - 27a^2d^2.
$$
The polynomial $p(x)$ has a multiple root if and only if $\Delta_{p}=0$. Moreover, if we denote by $\Delta_{p}'$ the quantity
$$
\Delta_{p}'=b^2 - 3ac,
$$
we have that $p$ has a triple root if and only if $\Delta_{p}=\Delta_{p}'=0$.

Let us begin our analysis of $Disc(\mathcal C)$ by the set of triple points.

\begin{prop}\label{triple}
Let $\mathcal{C}\subset \mathbb{CP}^3$ be the cubic of equation~\eqref{scroll}. Then the set of triple points $G_{1}$ of $\mathcal{C}$
is empty.
\end{prop}
\begin{proof}
Having the explicit parameterisation $\mathcal R$, for any $q\in \HH$, we compute, according to Formula \eqref{determinantA}, the associated determinant 
$$
D_{q}(v)=v^{3}-v^{2}\bar q_{1}+vq_{1}-|q_{1}|^{2}-|q_{2}|^{2}.
$$
For this polynomial the two quantities $\Delta_{D_{q}}$ and $\Delta_{D_{q}}'$ are respectively
$$
-\Delta_{D_{q}}= 4q_1^3+8|q_1|^4+4\bar q_1^3|q_1|^2+36|q_1|^2|q_2|^2+4\bar q_1^3|q_2|^2+27|q_2|^4
,\textrm{\ and \ } \Delta_{D_{q}}'=\bar q_{1}^{2}-3q_{1}.
$$
Decomposing the two equations $\Delta_{D_{q}}=0$ and $\Delta_{D_{q}}'=0$ in their real and imaginary parts 
and setting $z:=|q_2|$ and $q_1=x+iy$, we obtain the following two systems:
\begin{equation}\label{systemdisc}
\Delta_{D_{q}}=0 \Leftrightarrow
  \begin{cases}
  27z^4+z^2(36(x^2+y^2)+4x(x^2-3y^2))
  +4x(x^2-3y^2)(x^2+y^2+1)+8(x^2+y^2)^2=0\\
  4y(3x^2-y^2)(1-(x^2+y^2)-z^2)=0.
 \end{cases}
\end{equation}
and
\begin{equation*}
\Delta_{D_{q}}'=0 \Leftrightarrow
 \begin{cases}
  x^{2}-y^{2}-3x=0\\
  2xy+3y=0.
   \end{cases}
\end{equation*}
The last system represents the intersection of a hyperbola with the product of two lines (see Figure~\ref{hyperbola}).
\begin{figure}[ht]
\centering
\includegraphics[width=0.3\linewidth, height=5cm]{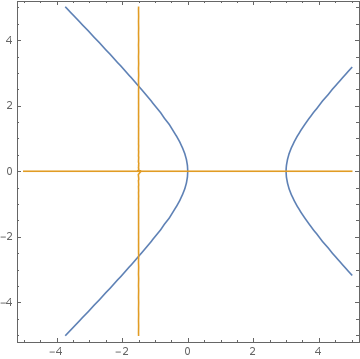}
\caption{The set $\Delta_{D_{p}}'=0$ is given by the intersection of the blue
hyperbola $x^{2}-y^{2}-3x=0$ with the two yellow lines  
$2xy+3y=0$.}
\label{hyperbola}
\end{figure}
The solutions of the equation $\Delta_{D_{p}}'=0$ are then $q=q_{2},3+q_{2},-3/2\pm3\sqrt{3}/2i+q_{2}$
for any $q_{2}\in\C$.

We study the set of triple points by substituting these four solutions in System \eqref{systemdisc}.
\begin{itemize}
\item For $q_{1}=0$, the discriminant $\Delta_{D_q}$ vanishes only for $q_{2}=0$, but the point $q=0$
belongs to $G_{0}$.
\item For $q_{1}=3$, the equation $\Delta_{D_{q}}=0$ becomes 
$$
27 z^4+432 z^2+1728=0,
$$
that has no real solution.
\item For $q_{1}=-3/2\pm3\sqrt{3}/2i$, it holds $|q_{1}|=3$ and $(-3/2\pm3\sqrt{3}/2i)^{3}=27$,
therefore, in both cases, the equation $\Delta_{D_{q}}=0$ becomes, as before, 
$$
27 z^4+432 z^2+1728=0,
$$
which, again, has no real solution.
\end{itemize}
Therefore, we have proved that $G_{1}=\emptyset$.
\end{proof}

We now pass to the study of twistor fibers. 

\begin{prop}\label{fibre}
Let $\mathcal{C}\subset \mathbb{CP}^3$ be the cubic of equation~\eqref{scroll}. Then the set of twistor fibres $G_{0}$ of $\mathcal{C}$
is equal to
$$
G_0=\left\{0,-1,\frac{1}{2}+i\frac{\sqrt{3}}{2},\frac{1}{2}-i\frac{\sqrt{3}}{2},\infty\right\}\subset\hat{\C}_i.$$
\end{prop}
\begin{proof}
In Remark~\ref{lines} we have already seen that $\{0,-1,\frac{1}{2}+i\frac{\sqrt{3}}{2},\frac{1}{2}-i\frac{\sqrt{3}}{2},\infty\}\subset G_{0}$. We now show that there are no other elements in $G_{0}$ using the parameterisation given by $\mathcal{R}$.
For any $q\in \HH$, the matrix $A_q$ associated with $\mathcal R$, is given by
\begin{equation*}
	A_q(v)=
\begin{pmatrix}
	 -v^{2}-q_1 &  -\bar q_2\\
	-q_2 & \bar q_1-v
\end{pmatrix},
\end{equation*}
therefore $A_{q}\equiv 0$ if and only if $q_{2}=0$ and the following system of equations
\begin{equation*}
\begin{cases}
 -v^{2}-q_1=0\\
  \bar q_1-v=0,
\end{cases}
\end{equation*}
is satisfied, i.e. if and only if $q_{1}+\bar q_{1}^{2}=0$ and $v^{2}+\bar v=0$. The first of the two equations gives
$q_{1}\in\{0,-1,1/2\pm\sqrt{3}/2\}$ and the same values for $v$. 
Thanks to Theorem~\ref{systempre-images}, for any $q_{1}\in\{0,-1,1/2\pm\sqrt{3}/2\}$ the system related to the matrix $A_{q}$ has a set of solutions biholomorphic to $\mathbb{CP}^{1}$, implying that $\{0,-1,1/2\pm\sqrt{3}/2\}\subset G_{0}$. The determinant of $A_{q}$ is given (as in proof of Proposition~\ref{triple}) by
$$
D_{q}(v)=v^{3}-v^{2}\bar q_{1}+vq_{1}-|q_{1}|^{2}-|q_{2}|^{2}.
$$
Since $D_{q}$  is never the zero polynomial, then the only other element that possibly lie in $G_{0}$ is $\infty$
which was already found in Remark~\ref{lines}.
\end{proof}

We conclude the study of $Disc (\mathcal C)$ by describing the set of double points.
\begin{prop}\label{double}
Let $\mathcal{C}\subset \mathbb{CP}^3$ be the cubic of equation~\eqref{scroll}. Then the set of double points $G_{2}$ of $\mathcal{C}$
is homeomorphic to a $2$-sphere with six handles pinched at $G_{0}$ without the attaching points, i.e.
$$
G_{2}=\left(\hat{\C}_i\bigcup_{\substack{P=0,\infty\\ z_m^3=-1,\, z_m\in\C_i}}\Sigma_{P,z_m}\right)\setminus G_0,
$$
where $\Sigma_{P,z_m}$ denotes a handle pinched to $\hat{\C}_i$ at $P$ and $z_m$. 
\end{prop}
\begin{proof}

First of all, thanks to Remark~\ref{lines}, the set $(\C_{i}\setminus G_0)\subset G_{2}$. Then, since the set of triple points is empty, it is sufficient to study the zero set of the discriminant of
$D_{q}$, i.e. the set of solutions of System~\eqref{systemdisc}.
Since the second equation of the system factorises into three simpler parts, let
us study them separately.

\noindent
\textbf{First case: $y=0$.} We get the following equation
\begin{align}
p_{1}(x,z)= 
27z^{4}+z^{2}[4x^{2}(4+x)]+4x^{3}(x+1)^{2}=0.\label{eqquinti3}
\end{align}
This equation defines a curve which intersects the $x$-axis only at $(x,z)=(0,0),(-1,0)$ (and at $\infty$).
By Equation~\eqref{eqquinti3}, it is not difficult to see that the curve is contained in the half-plane $x\leq 0$.
Now for $x\leq 0$ we can solve Equation~\eqref{eqquinti3} for $z^{2}$, obtaining
$$
z^{2}=\frac{2}{27}(-x^{2}(4+x)\pm\sqrt{x^{4}(4+x)^{2}-27x^{3}(x+1)^{2}}).
$$
Since for $x\leq0$,  $-x^{2}(4+x)\leq\sqrt{x^{4}(4+x)^{2}-27x^{3}(x+1)^{2}}$ and equality holds for
$x=0,-1$, then, the only admissible solution is
$$
z^{2}=\frac{2}{27}(-x^{2}(4+x)+\sqrt{x^{4}(4+x)^{2}-27x^{3}(x+1)^{2}})=f(x),
$$
For $x\leq 0$, the function $f(x)$ appearing on the right hand side of the previous equation is always non-negative, differentiable for any $x\neq 0, -1$ where it is equal to $0$ and it diverges at $+\infty$ for $x$ that goes to $-\infty$.
The curve $p_{1}(x,z)=0$ is then given by the union of the graphs of $\pm\sqrt{f(x)}$, for $x\leq 0$ (see Figure~\ref{fig:subim1}).
\begin{figure}[h]
\includegraphics[width=0.35\linewidth]{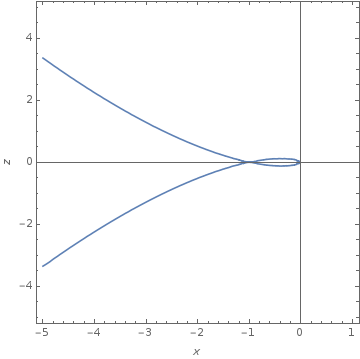} 
\caption{$\{p_{1}(x,z)=0\}$}
\label{fig:subim1}
\end{figure}

\noindent
\textbf{Second case: $y^{2}=3x^{2}$.} The set of points satisfying $y^{2}-3x^{2}=0$ is the union of the two lines $y=\pm\sqrt{3}x$ in the $xy$-plane. These are the images of the line $y=0$ under the rotations around the $z$-axis of $2\pi/3$ and of $4\pi/3$. Notice that the first equation in System~\eqref{systemdisc}
is invariant under these two rotations. Therefore the set of solutions of $\Delta_{p}=0$ for $y^{2}-3x^{2}=0$
consists exactly of two copies of the curve $p_{1}(x,z)=0$ obtained by rotations of angles $2\pi/3$ and $4\pi/3$ around the $z$-axis. In particular these curves intersect the $xy$-plane at the origin and at 
$(1/2,-\sqrt{3}/2)$ and $(1/2,\sqrt{3}/2)$, respectively.

\noindent
\textbf{Third case: $y^{2}=1-x^{2}-z^{2}$.} In this case, System~\eqref{systemdisc} reduces to 
$$
z^{4}-4(5+6x)z^{2}-8(4x^{3}-3x+1)=0.
$$
Solving for $z^{2}$ we get
$$
z^{2}=2(5+6x\pm\sqrt{(3+2x)^{3}}).
$$
For any $|x|<1$, $5+6x-\sqrt{(3+2x)^{3}}\leq0$ and equals zero only for $x=1/2$, giving again the points $(1/2,\pm \sqrt{3}/2)$ in the $xy$-plane.
On the other hand  $0\leq 5+6x+\sqrt{(3+2x)^{3}}\leq1-x^{2}$ if and only if $x=-1$, corresponding to the already found point $(-1,0)$ in the $xy$-plane.
Thus this case does not give any new contribution to the set of double points.

Recall now that the variable $z$ equals $|q_{2}|$.  Therefore the three curves found previously in the first and second cases, represent in fact three surfaces in $\SF^{4}$: $p_1(x,|q_2|)=0$ is hence a sphere with an equator collapsed to a point, 
intersecting $\hat{\C}_{i}$ only at $0, -1$ and at $\infty$. 
Removing these three points its connected components (two handles) are both homeomorphic to a cylinder. The other four handles are obtained by rotation around the $z$-axis of $2\pi/3$ and $4\pi/3$. 


\end{proof}

Collecting the results contained in Remark~\ref{lines} and in Propositions~\ref{triple}, \ref{fibre} and~\ref{double},
we can state the main theorem of this section.

\begin{theorem}\label{discriminant}
The discriminant locus of the cubic scroll $\mathcal{C}\subset \mathbb{CP}^3$ of equation~\eqref{scroll} is given by the singular surface 
$$
Disc(\mathcal{C})=\Sigma=\hat{\C}_i\bigcup_{\substack{P=0,\infty\\ z_m^3=-1, z_m\in\C_i}}\Sigma_{P,z_m},
$$
where $\Sigma_{P,z_m}$ denotes an handle pinched to $\hat{\C}_i$ at $P$ and $z_m$. 
In particular we have that 
$$
G_0=\left\{0,-1,\frac{1}{2}+i\frac{\sqrt{3}}{2},\frac{1}{2}-i\frac{\sqrt{3}}{2},\infty\right\}\subset\C_i,\quad G_1=\emptyset,\quad G_2=\Sigma\setminus G_0.
$$
\end{theorem}

\begin{figure}[ht]
\centering
\includegraphics[scale=0.35]{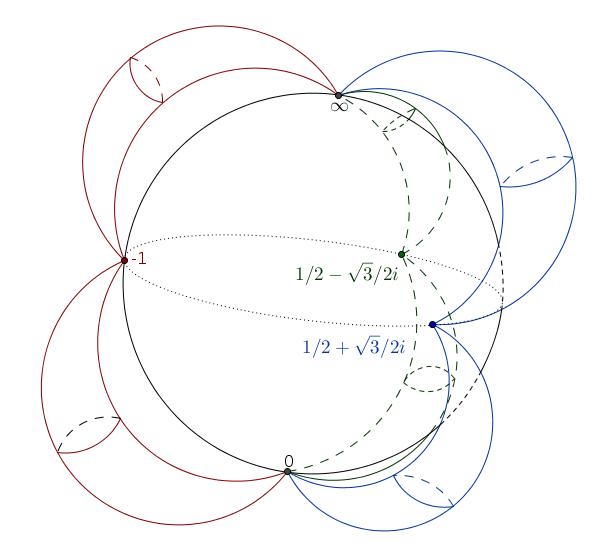}
\caption{The set $\Sigma\subset\SF^4$. The big black round sphere represents $\hat{\C}_{i}$. The two red handles represent $\Sigma_{0,-1}$ and $\Sigma_{\infty,-1}$. The two green handles represent $\Sigma_{0,1/2-\sqrt{3}/2}$ and $\Sigma_{\infty,1/2-\sqrt{3}/2}$. The two blue handles represent $\Sigma_{0,1/2+\sqrt{3}/2}$ and $\Sigma_{\infty,1/2+\sqrt{3}/2}$}
\end{figure}

\subsection{A first description of the OCS's induced by $\mathcal{C}$}
We will now give some qualitative results on the nature of the 3 OCS's induced by the cubic scroll $\mathcal{C}$.

A slightly different approach in twistor theory relies on the definition of twistor fibers by means
of a map $j$ which swaps any OCS $J$ into $-J$. In our case $j$ is the transformation defined on 
$\mathbb{CP}^{3}$ induced, via $\pi$, by the quaternionic multiplication by $j$ on $\mathbb{HP}^{1}$.
To be more precise $j:\mathbb{CP}^{3}\to\mathbb{CP}^{3}$ is the map defined as
\begin{equation*}
j:[X_{0},X_{1},X_{2},X_{3}]\mapsto[-\bar X_{1},\bar X_{0},-\bar X_{3},\bar X_{2}].
\end{equation*}
This map $j$ is an antiholomorphic involution and it has no fixed points. 
Starting with such a map $j$, one can recover the twistor fibres in the following way: consider a point $X\in\mathbb{CP}^{3}$
and the unique projective line $l$ connecting $X$ and $j(X)$. All the lines constructed in this way form the set of twistor fibres. In particular, a line $l$ in $\mathbb{CP}^{3}$ is a fibre for $\pi$ if and only if $l=j(l)$.

%
%
Thanks to the explicit description of the map $j$, we are able to give a first
qualitative result on the three OCS's induced by the cubic $\mathcal{C}$.

\begin{prop}\label{jC}
Let $\mathcal{C}\subset \mathbb{CP}^3$ be the cubic scroll of equation~\eqref{scroll}. Then the intersection of $\mathcal C$
with $j(\mathcal C)$ consists of the following set
\begin{equation*}
j(\mathcal{C})\cap\mathcal{C}= m_{02}\cup m_{13} \cup \pi^{-1} (G_{0}).
\end{equation*}
where, in particular $j(m_{02})=m_{13}$. 
\end{prop}
\begin{proof}
By the defintion of the map $j$, we have that $[X_0,X_1,X_2,X_3]\in j( \mathcal{C})\cap\mathcal{C}$ if and only if 
\begin{equation*}
\begin{cases}
X_{0}X_{3}^{2}+X_{1}^{2}X_{2}=0\\
X_{0}^{2}X_{3}+X_{1}X_{2}^{2}=0.
\end{cases}
\end{equation*}
It is immediate to verify that the lines $m_{01}, m_{02}, m_{13}$ and $m_{23}$ belong to the intersection $j(\mathcal{C})\cap\mathcal{C}$.
The fact that $j(m_{02})=m_{13}$, $\pi(j(m_{02}))=\pi(m_{13})=\hat{\C}_{i}$, $\pi(m_{01})=\infty\in\SF^{4}$ and $\pi(m_{23})=0$ is trivial.
Suppose now that $X_{n}\neq 0$ for all $n=0,1,2,3$, then from the first equation we have that
\begin{equation}\label{eq1}
X_{0}=-\frac{X_{1}^{2}X_{2}}{X_{3}^2}.
\end{equation}
Hence, from the second equation we obtain
\begin{equation*}
X_{1}X_{2}^{2}[X_{3}^{3}+X_{1}^{3}]=0\quad\Leftrightarrow\quad X_{3}^{3}=-X_{1}^{3}. 
\end{equation*}
Therefore we obtain the three planes defined by the following equations
\begin{equation*}
X_{3}=-X_{1},\qquad X_{3}=\left(\frac{1}{2}-\frac{\sqrt{3}}{2}i\right)X_{1},\qquad X_{3}=\left(\frac{1}{2}+\frac{\sqrt{3}}{2}i\right)X_{1}.
\end{equation*}
Substituting the last three equations in Formula~\eqref{eq1}, we get the three lines projecting
on $-1, \frac{1}{2}\pm\frac{\sqrt{3}}{2}i\in\SF^{4}$ via $\pi$:
\begin{equation*}
m_{-1}:\begin{cases}
X_{3}=-X_{1}\\
X_{2}=-X_{0}
\end{cases},\quad m_{\frac 1 2-\frac {\sqrt{3}}{ 2}i}:\begin{cases}
X_{3}=\left(\frac{1}{2}-\frac{\sqrt{3}}{2}i\right)X_{1}\\
X_{2}=-\left(\frac{1}{2}-\frac{\sqrt{3}}{2}i\right)^{2}X_{0}
\end{cases}
,\quad m_{\frac 1 2+\frac {\sqrt{3}}{ 2}i}:\begin{cases}
	X_{3}=\left(\frac{1}{2}+\frac{\sqrt{3}}{2}i\right)X_{1}\\
	X_{2}=-\left(\frac{1}{2}+\frac{\sqrt{3}}{2}i\right)^{2}X_{0}.
\end{cases}
\end{equation*}

\end{proof}
A picture of the set $j(\mathcal{C})\cap\mathcal{C}$ is given in Figure~\ref{jay}, where it is also highlighted
its twistor projection.

\begin{figure}[ht]
\includegraphics[width=0.3\linewidth]{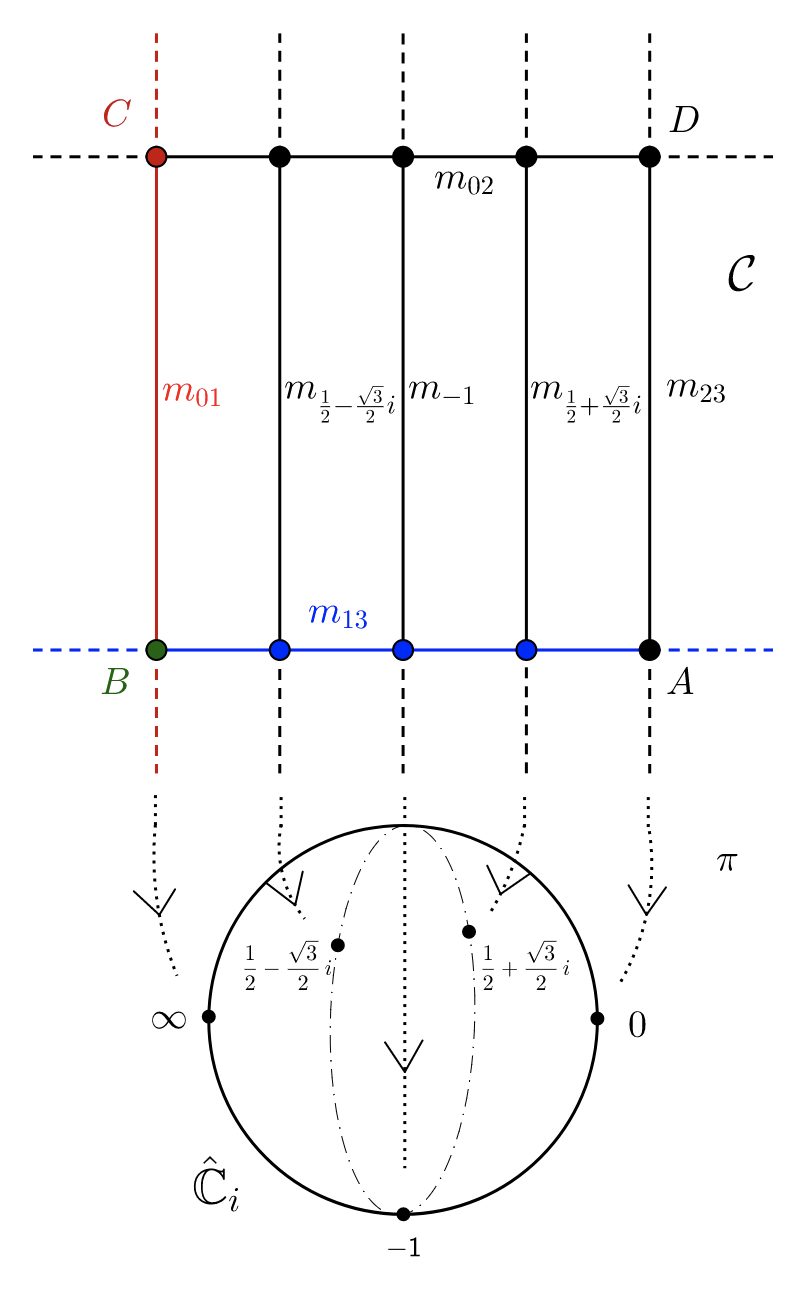}
\caption{The set $j(\mathcal{C})\cap\mathcal{C}$ and its twistor projection.}
\label{jay}
\end{figure}

In the language of independent OCS's, the last proposition has as main implication the following result.


\begin{cor}\label{3ocs}
The manifold $\SF^{4}\setminus \Sigma$ admits 3 non-constant strongly independent OCS's. 
\end{cor}

\begin{proof}
Thanks to Theorem~\ref{discriminant}, the cubic $\mathcal{C}$ induces three
OCS's on $\SF^{4}$ outside $\Sigma$. Moreover, thanks to Proposition~\ref{jC}, outside of $m_{02}\cup m_{13}=\pi^{-1}(\hat{\C}_{i})$, the intersection between $\mathcal{C}$ and $j(\mathcal{C})$ is empty, meaning that
if $J$ is an OCS induced by $\mathcal{C}$, then $-J$ (that would be obtained by $j(\mathcal{C})$),
is none of the other 2 possible OCS's induced by the cubic.
\end{proof}

Thanks to~\cite[Theorem 22]{AAtwistor}, if $q=q_{1}+q_{2}j\in\SF^{4}\setminus\Sigma$ and if 
$$
\{x_{n}=\alpha_{n}+I_{n}\beta_{n}, \, n=1,2,3\}=\pi(\mathcal{R}^{-1}(\pi^{-1}(q))),
$$
then the three OCS's are defined to be
$$
\mathbb{J}_{n,q}(v)=I_{n}v,
$$
where $v\in T_q(\SF\setminus\Sigma)\simeq\HH$ and $I_{n}v$ denote the quaternionic multiplication. 
It is in general difficult to write an explicit expression of the OCS's $\mathbb{J}_{n}$, but at 
a fixed point $q$ the matrix representing $\mathbb{J}_{n,q}$ can be computed as in~\cite[Remark 1]{AAtwistor}, by knowing the $[s,u]$-coordinate of the element $\mathcal{R}^{-1}(\pi^{-1}(q))$.
%
%

\begin{remark}
As a consequence of Proposition~\ref{jC}, on the discriminant locus $\Sigma=Disc(\mathcal{C})$ the three OCS's of Corollary~\ref{3ocs}
collapse into two that are independent outside $\hat{\C}_{i}$ where they coincide with the two
constant OCS's $\pm\mathbb{J}_{i}$ defined as multiplication by $\pm i$.
\end{remark}

	The non-singular cubic 
	$$
	\mathcal{C}':\quad X_0X_3^2+X_3X_0^2+X_1X_2^2+X_2X_1^2=0
	$$ 
	studied in~\cite{APS, AS} is engineered so to have the ``twistor geometry [...] as simple as possible''  among all the non-singular cubics (see~\cite{AS}).
	This is done by maximising the number of twistor lines on a generic non-singular cubic surface $\mathcal{K}\subset\mathbb{CP}^3$ and the
	the size of the group of symmetries preserving $\mathcal{K}$ and the twistor fibration.
	Thanks to~\cite[Theorem 1]{APS}, a non-singular cubic surface contains at most five twistor lines. Moreover, in the same paper 
	the authors prove the following facts:
	\begin{itemize}
		\item if a smooth cubic contains $5$ twistor lines, then the image of these lines under $\pi$ must all lie
		on a round $2$-sphere in $\SF^4$;
		\item given a line $m\in\mathbb{CP}^3$ and five points $\{p_1,\dots, p_5\}\subset m$, then,
		after having identified $m$ with the Riemann sphere,
		there exists a non-singular cubic surface containing $m$, $j(m)$ and the five twistor lines joining
		$p_i$ and $j(p_i)$ if and only if no four of the points $p_i$ lie on a circle under this identification.
	\end{itemize}
	Now, in~\cite[Section 5]{APS} the authors show that the the most symmetrical arrangement of $5$ points on
	a $2$-sphere such that no four of which lie on a circle, is given by the set of vertices of a triangular bipyramid:
	two poles and three cubic roots of $-1$ lying on the equator. The symmetry group of these five points is $\mathbb{Z}_3\times\mathbb{Z}_2$
	and up to conformal transformation of the $2$-sphere, it is possible to assume that these points are the elements of $G_0$ listed in our Proposition~\ref{lines}.
	Then if the $2$-sphere $\hat{\C}$ containing $G_0$, lifts, by means of $\pi$, to the two lines $m$ and $j(m)$ and
	we ask for the surface $\mathcal{K}$ to contain both $m$ and $j(m)$, then its symmetry group is give by 
	$\mathbb{Z}_3\times\mathbb{Z}_2\times\mathbb{Z}_2$, where the action of the last $\mathbb{Z}_2$ swaps $m$ and $j(m)$.
	Up to conformal transformations, the unique non-singular cubic surface having symmetry group $\mathbb{Z}_3\times\mathbb{Z}_2\times\mathbb{Z}_2$ is exactly $\mathcal{C}'$.
	The discriminant locus of this non-singular surface topologically is $(T_1\cup T_2)/ G_0$, where $T_1$ and $T_2$ are two $2$-dimensional
	tori passing trough $G_0$.
	
	It is interesting to notice that our example carries on part of this geometry. In fact, Theorem~\ref{discriminant} states that the symmetry group of the cubic scroll 
	$\mathcal{C}$ contains $\mathbb{Z}_3\times\mathbb{Z}_2$ acting on the twistor lines but not the second $\mathbb{Z}_2$ part.
	In fact, while $j(\mathcal{C}')=\mathcal{C}'$, in our case the intersection $j(\mathcal{C})\cap\mathcal{C}$
	is non trivial. 


The construction that we presented and that develops the fascinating interplay between slice regularity and twistor geometry, still has some
aspects that need to be better understood.
Given a slice-polynomial function $P$ (or more generally a slice regular function), the relation between its singular set (i.e.: the set of points where its differential is not injective) and the discriminant locus of the surface containing the image of its extended twistor lift it is not completely clear.
In fact, given a slice regular function $f$, the sets of its wings and degenerate spheres are contained in the singular locus of $f$ but while 
degenerate spheres correspond to 
twistor fibres contained in the
image of $\mathcal{P}$, on the other hand the same is not true, in general, for wings.
Consider for instance the slice-polynomial function $P(q)=(q+j)*\ell_+=\ell_+v+\ell_-j$.
As shown in Example~\ref{esempicompleti} we have that $P|_{\C_{i}^{-}}=j$. Its extended twistor lift is equal to
$\mathcal{P}([s,u],v)=[s,u,sv-u,0]$ and its image lies in the hyperplane $\mathcal{H}:=\{X_{3}=0\}$. The hyperplane $\mathcal{H}$ was extensively studied in~\cite{AAtwistor}.
In particular $\#(\pi^{-1}(j)\cap\mathcal{H})=1$ while $\pi^{-1}(0)\cap\mathcal{H}\simeq \mathbb{CP}^{1}$. In fact the pre-image of $j$ by $\pi\circ\mathcal{P}$
consists of $([0,1],v)\subset\mathbb{CP}^{1}\times \C_{i}$, while the pre-image of $0$ consists of the surface $([1,v],v)\subset\mathbb{CP}^{1}\times \C_{i}$ for any $v\in\C_{i}$. From this point of view, $j$ is not
a twistor fiber because the projection of its pre-image on the first component is just the point $[0,1]$,
while the projection of the pre-image of $0$, $[1,v]$, covers $\mathbb{CP}^{1}$ minus a point.

%
%
%
%
%
%
%
%

\bibliographystyle{amsplain}

\end{document}